\documentclass[12pt]{article}
\usepackage{geometry}                
\usepackage{graphicx}
\usepackage{amssymb}
\usepackage{amsmath}
\usepackage{mathtools}
\usepackage{mathrsfs}
\usepackage{xcolor}
\usepackage{amsthm}
\usepackage{hyperref}
\newcommand*{\FF}{\mathbb{F}}
\newcommand*{\NN}{\mathbb{N}}
\newcommand*{\ZZ}{\mathbb{Z}}

\newcommand*{\RR}{\mathbb{R}}

\newcommand*{\PP}{\mathbb{P}}
\DeclareMathOperator*{\EE}{\mathbb{E}}

\newcommand*{\calE}{\mathcal{E}}
\newcommand*{\calF}{\mathcal{F}}

\newcommand*{\calO}{\mathcal{O}}
\newcommand*{\calU}{\mathcal{U}}
\newcommand*{\calV}{\mathcal{V}}

\newcommand*{\calX}{\mathcal{X}}

\newcommand*{\A}{\mathtt{A}}

\newcommand*{\1}{\mathbf{1}}

\newcommand*{\mb}[1]{\mathbf{#1}}
\newcommand*{\bOmega}{\mb{\Omega}}

\newcommand*{\st}{\,:\,}

\newcommand*{\ball}[3][\relax]{\mathrm{B}^{#1}(#2, #3)}

\newtheorem{lemma}{Lemma}[section]
\newtheorem{cor}[lemma]{Corollary}
\newtheorem{prop}[lemma]{Proposition}
\newtheorem{theorem}[lemma]{Theorem}
\newtheorem{conj}[lemma]{Conjecture}

\newtheorem{mainthm}{Theorem}

\theoremstyle{definition}
\newtheorem{example}[lemma]{Example}

\newtheorem{defn}[lemma]{Definition}

\DeclareMathOperator{\conv}{conv}
\DeclareMathOperator{\Prob}{Prob}

\DeclareMathOperator{\Unif}{Unif}

\DeclareMathOperator{\Hom}{Hom}

\DeclareMathOperator{\Sym}{Sym}

\DeclareMathOperator{\diam}{diam}
\DeclareMathOperator{\shent}{H}
\DeclareMathOperator{\h}{h}

\DeclareMathOperator{\f}{f}

\DeclareMathOperator{\press}{p}
\DeclareMathOperator{\Press}{P}

\DeclareMathOperator{\ann}{mod}
\DeclareMathOperator{\arctanh}{arctanh}

\let\mbf\mathbf

\newcommand*{\KS}{\mathrm{KS}}
\newcommand*{\edge}{\mathrm{edge}}
\newcommand*{\FB}{\mathrm{FB}}

\DeclareMathOperator{\lwlim}{lwlim}

\DeclarePairedDelimiter{\abs}{\lvert}{\rvert}

\newcommand{\nnorm}[1]{{\left\vert\kern-0.25ex\left\vert\kern-0.25ex\left\vert #1 
    \right\vert\kern-0.25ex\right\vert\kern-0.25ex\right\vert}}
\DeclarePairedDelimiterX{\inprod}[2]{\langle}{\rangle}{#1,\ #2}

\newcommand*{\good}{{good}}

\title{Equilibrium and nonequilibrium Gibbs states on sofic groups}
\author{Christopher Shriver}

\begin{document}
\maketitle

\begin{abstract}
	Recent work of Barbieri and Meyerovitch has shown that, for very general spin systems indexed by sofic groups, equilibrium (i.e. pressure-maximizing) states are Gibbs. The main goal of this paper is to show that the converse fails in an interesting way: for the Ising model on a free group, the free-boundary state typically fails to be equilibrium as long as it is not the \emph{only} Gibbs state. For every temperature between the uniqueness and reconstruction thresholds a typical sofic approximation gives this state finite but non-maximal pressure, and for every lower temperature the pressure is non-maximal over \emph{every} sofic approximation.
	
	We also show that, for more general interactions on sofic groups, the local limit of Gibbs states over a sofic approximation $\Sigma$, if it exists, is a mixture of $\Sigma$-equilibrium states. We use this to show that the plus- and minus-boundary-condition Ising states are $\Sigma$-equilibrium if $\Sigma$ is any sofic approximation to a free group. Combined with a result of Dembo and Montanari, this implies that these states have the same entropy over every sofic approximation.
\end{abstract}

\tableofcontents

\section{Overview}
Consider a shift system $\A^\Gamma$, where $\A$ is a finite set and $\Gamma$ is a countable sofic group acting on $\A^\Gamma$ by the ``shift'' action
	\[ (g \cdot \mb{x})_h = \mb{x}_{g^{-1} h} . \]
Suppose $u \colon \A^\Gamma \to \RR$ is a continuous function, which we think of as specifying for each $\mb{x} \in \A^\Gamma$ the ``specific energy at the identity $e \in \Gamma$.'' Also let $\h_\Sigma \colon \Prob^\Gamma(\A^\Gamma) \to [-\infty, \log\abs{\A}]$ denote the sofic entropy with respect to some sofic approximation $\Sigma$ to $\Gamma$.

We consider two standard notions of what it means for a probability measure $\mu \in \Prob^\Gamma(\A^\Gamma)$ to describe the state of a system in statistical equilibrium. First, we say that $\mu$ is a $\Sigma$-\emph{equilibrium state} if it maximizes the \emph{sofic pressure}
	\[ \press_\Sigma(\mu) = \h_\Sigma(\mu) - \int u\, d\mu. \] 
Definitions of similar quantities often include a factor of temperature or inverse temperature on one of the terms; we absorb this factor into the function $u$. See Section \ref{sec:definitions} for more precise definitions.

The second notion is that of a \emph{Gibbs state}. Briefly, a measure $\mu \in \Prob(\A^\Gamma)$ is a Gibbs state with respect to some potential if it has certain local conditional probabilities specified by that potential. The collection of these conditional probabilities is called a specification, or the DLR equations, after Dobrushin, Lanford, and Ruelle who initially developed the concept. They can be interpreted as a kind of local equilibrium criterion. We will not provide a precise definition of ``Gibbs state'' here; instead we will work with specific standard examples. For precise definitions see for example \cite{georgii2011}. When we say a measure ``is Gibbs'' or ``is a Gibbs state,'' we always mean with respect to a particular potential. 

When $\Gamma$ is amenable and the energy function $u$ is ``nice enough,'' the set of equilibrium states is the same as the set of Gibbs states. The first versions of this theorem are due to Dobrushin, Lanford, and Ruelle. For $\ZZ^d$-indexed finite-alphabet shift systems with $u$ defined by an ``absolutely summable potential,'' see \cite[Chapter 15]{georgii2011}. For general amenable groups, see \cite{tempelman1984}.

One direction of this equivalence has been extended to actions of sofic groups: Barbieri and Meyerovitch have recently shown that every equilibrium state is Gibbs as long as $\Gamma$ is sofic and $u$ satisfies a somewhat technical (but very general) criterion \cite{barbieri2022a}.

The most natural version of a converse would be that every Gibbs state is equilibrium for ``nice enough'' interactions. It is not clear what ``nice enough'' should mean exactly, but it is natural to assume it should include the Ising model: this is generally seen as the simplest nontrivial interaction. We will consider, then, whether every Gibbs state for an Ising model on the Cayley graph of a free group $\Gamma$ (\emph{i.e.} an infinite regular tree) must be an equilibrium state.

But, if the temperature is low enough, there exists a sofic approximation $\Sigma$ to the free group $\Gamma$ such that the free-boundary Ising state $\mu^{\FB}$ (which is Gibbs) has $\h_\Sigma(\mu^{\FB}) = -\infty$, so $\press_\Sigma(\mu^{\FB}) = -\infty$. Intuitively, this is because a good model for $\mu^{\FB}$ on a sofic approximation is a bipartition, which has very small cut density if the temperature is low. Therefore a typical sofic approximation does not admit good models for $\mu^{\FB}$. But any Bernoulli shift, for example, will have finite $\press_\Sigma(\cdot)$, so $\mu^{\FB}$ is a Gibbs state which is not $\Sigma$-equilibrium.

This sort of degenerate behavior of sofic entropy, where $\Sigma$ is simply incompatible with $\mu^{\FB}$, is well-known and does not provide a very satisfying counterexample. The same example can be used to show that there are two sofic approximations which give $\mu^{\FB}$ different sofic entropy values. But it is still an open question whether a measure can have different \emph{finite} sofic entropy values.

This suggests looking for a partial converse: that every Ising Gibbs state with \emph{finite} $\Sigma$-entropy is $\Sigma$-equilibrium. In this paper we show that this is not true. We show that for all temperatures below the uniqueness threshold (and no external field) the free-boundary state is $\Sigma$-non-equilibrium for some $\Sigma$ for which $\h_\Sigma(\mu^{\FB}) > - \infty$. For temperatures above the reconstruction threshold, this is true for a ``typical'' $\Sigma$.

Furthermore, we show that for low enough temperatures $\mu^{\FB}$ is nonequilibrium over \emph{every} $\Sigma$. This also rules out the possibility that some additional assumption on the sofic approximation could make the converse hold, or that every Gibbs state is equilibrium over \emph{some} sofic approximation.

Specifically, for free groups of rank at least 2, we establish that $\mu^{\FB}$ is non-equilibrium over \emph{typical} $\Sigma$ for all temperatures between the uniqueness and reconstruction thresholds (Theorem \ref{thm:existence}). The temperature being below either the reconstruction threshold or half the uniqueness threshold is a sufficient condition for $\mu^{\FB}$ to be non-equilibrium over \emph{every} $\Sigma$ (Theorem \ref{thm:noneq}); it is unclear whether this is true up to the uniqueness threshold. See Figure~\ref{fig:tempranges}. \\

\begin{figure}
\begin{center}
\includegraphics[width=\textwidth]{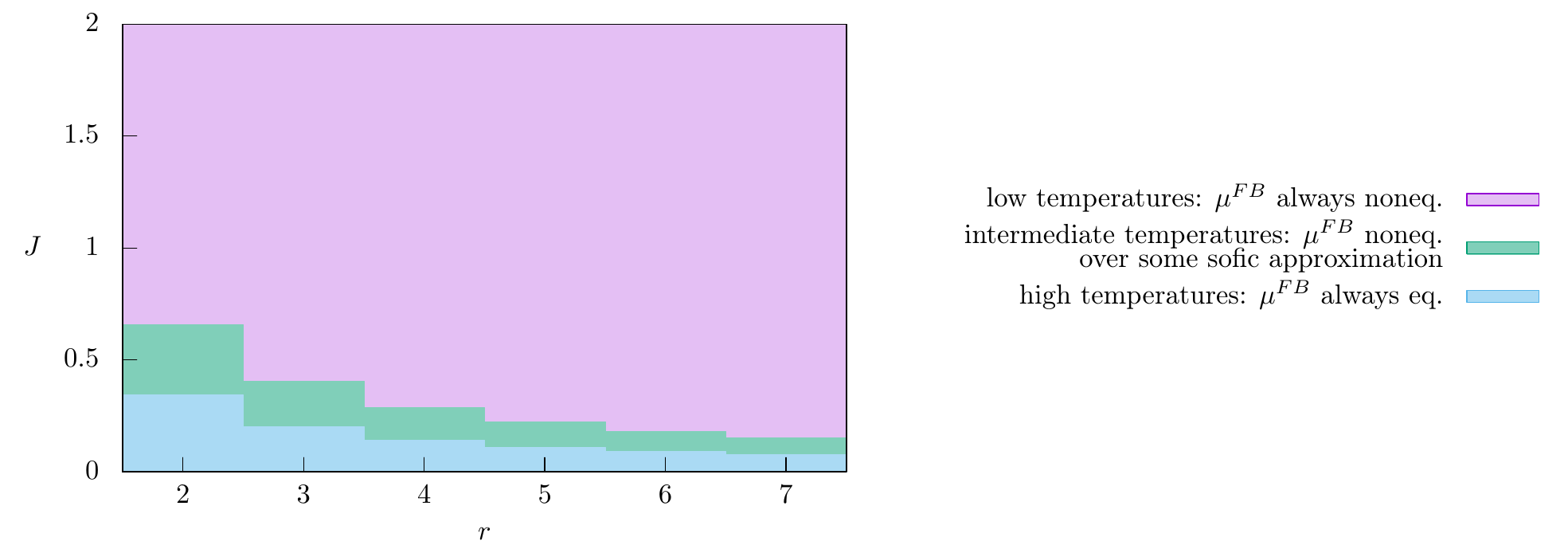}
\end{center}
\caption{Temperature ranges where equilibrium status of the free boundary state $\mu^{\FB}$ is established. The horizontal axis is $r$, the rank of the free group (and half the degree of the Cayley tree), and the vertical axis is inverse temperature or interaction strength. \\ In the upper region, $\mu^{\FB}$ is never equilibrium. In the middle region, it is nonequilibrium over some sofic approximations. In the lower region it is the only Gibbs state, so must be equilibrium.}
\label{fig:tempranges}
\end{figure}

We also develop a method for establishing that some states are equilibrium (Theorem \ref{thm:eqgeneral}). We use it (and \cite{montanari2012}) to show that the plus and minus states (with no external field) are equilibrium over every sofic approximation to a free group. Combined with \cite{dembo2010a}, this implies that the plus and minus Ising states on a free group have the same sofic entropy over every sofic approximation.

\subsection{Statements of main results}

The proof of Theorem~\ref{thm:existence}, that $\mu^{\FB}$ is nonequilibrium over a typical sofic approximation, is based on the following ``first moment'' calculation:

\begin{prop}
\label{prop:fnoneq}
	At every temperature below the uniqueness threshold, the free boundary state is $\f$-nonequilibrium.
\end{prop}

By ``$\f$-equilibrium'' we mean that a state maximizes the pressure where the entropy quantity is the $\f$ invariant. As shown in \cite{bowen2010}, the $\f$ invariant is the sofic entropy over a particular random sofic approximation: a sequence of uniformly random homomorphisms. In random graph terminology, this corresponds to the permutation model.

We prove Proposition~\ref{prop:fnoneq} using a natural one-parameter family $\{\mu_t \st t \in \RR\}$ of Markov chains in which $t=0$ gives the free boundary state. The uniqueness threshold is exactly where the second derivative of the $\f$-pressure at $t=0$ becomes positive. Since the free boundary state is a local minimum of the pressure within this family, it must be nonequilibrium.

This alone is not a satisfying counterexample because it is not obvious whether $\mu^{\FB}$ being $\f$-nonequilibrium is simply due to the degenerate behavior mentioned above: at low enough temperatures the $\f$ invariant, and hence the $\f$-pressure, is affected by the possibility of the random sofic approximation supporting no good models for $\mu^{\FB}$ at all. What we really want is a deterministic $\Sigma$ such that $\mu^{\FB}$ has finite $\h_\Sigma$ but is not $\Sigma$-equilibrium.

Moreover, Seward's ergodic decomposition formula for the $\f$-invariant \cite{seward2016} prevents any non-ergodic measure from being $\f$-equilibrium. If the ergodic decomposition of $\mu$ has Shannon entropy $\delta$, then (since the average specific energy is affine) $\press_{\f}(\mu)$ is $\delta$ less than the average of $\press_{\f}(\cdot)$ over the ergodic decomposition, so cannot be maximal. This is not a concern for $\mu^{\FB}$, since this state is ergodic at every temperature, but it is another reason why the converse can fail for somewhat degenerate reasons with random sofic approximations.

As long as the free-boundary state is tail-trivial (i.e. when the temperature is above the reconstruction threshold), we can use \cite{shriver2021} to show that a ``typical'' sofic approximation has the desired property:

\begin{mainthm}
\label{thm:existence}
	For every temperature between the uniqueness and reconstruction thresholds, for a ``typical'' sofic approximation $\Sigma$ the free boundary Ising state has finite $\Sigma$-entropy but is $\Sigma$-nonequilibrium.
	
	More precisely, there is a sequence of sets of homomorphisms $\{ S_n \subset \Hom(\FF_r, \Sym(n)) \}_{n=1}^\infty$ with $\PP(S_n) \to 1$ such that if $\Sigma = \{\sigma_n\}_{n=1}^\infty$ is a sequence of maps with $\sigma_n \in S_{m_n}$ for some $m_n \to \infty$ then
	\begin{enumerate}
		\item $\Sigma$ is a sofic approximation to $\FF_r$
		\item $\press_\Sigma(\mu^{\FB}) = \press_{\f}(\mu^{\FB}) < \press_{\f}(\mu^+) = \press_\Sigma(\mu^+)$
	\end{enumerate}
\end{mainthm}

The probability measure on $\Hom(\FF_r, \Sym(n))$ is the uniform distribution.

\emph{Remarks}:
\begin{enumerate}
	\item This result says something about the Boltzmann distribution over a typical large random regular graph: while there are microstates corresponding to the free-boundary state, the total probability of the set of these microstates is exponentially small at intermediate temperatures. At low enough temperatures, there are typically no free-boundary microstates. An analogous result for the Potts model is proven in \cite{coja-oghlan2023}. There, this phenomenon is called ``phase coexistence.''
	\item The main theorem of \cite{nam2022} implies that $\mu^{\FB}$ has good models over every sofic approximation, at all temperatures down to some multiple of the reconstruction threshold. In particular, the $\f$ invariant is unaffected by the possibility of having no good models. The second-moment method we use here instead applies over a wider range of temperatures, and also tells us that the $\Sigma$-pressure is exactly the $\f$-pressure for typical $\Sigma$.
\end{enumerate}

We then show that, for an overlapping range of temperatures, the free boundary state is never equilibrium over any sofic approximation:

\begin{mainthm}
\label{thm:noneq}
	For every temperature below either half the uniqueness threshold or below the reconstruction threshold, the free boundary state is $\Sigma$-nonequilibrium for every $\Sigma$.
\end{mainthm}

In the condition ``half the uniqueness threshold,'' the constant $\frac{1}{2}$ is not optimal. The hypothetical ideal constant would be $1$, but numerically it appears the the method used here can only achieve up to $\frac{1}{1.65} \approx 0.61$ . This is discussed further following the proof at the end of Section~\ref{sec:thmBproof}.

We also have a positive result showing that the plus and minus states are always equilibrium:

\begin{mainthm}
\label{thm:isingeq}
	With zero external field, the plus and minus Ising states are equilibrium over every sofic approximation to a free group. These two states are the only $\f$-equilibrium states.
\end{mainthm}

This is a consequence of a more general result (Theorem \ref{thm:eqgeneral}) which shows that, for any continuous energy function $u$, the local limit of finitary Gibbs states over a sofic approximation $\Sigma$ (if it exists) is a mixture of $\Sigma$-equilibrium measures.

Theorem \ref{thm:isingeq} implies uniqueness of sofic entropy for the plus/minus states:

\begin{cor}
	With zero external field, the plus and minus Ising states have the same sofic entropy over every sofic approximation.
\end{cor}
\begin{proof}
	\cite[Theorem 2.4]{dembo2010a} shows that the sofic topological pressure is the same over any deterministic sofic approximation. Since, by Theorem \ref{thm:isingeq}, the plus and minus states attain the supremum in Chung's variational principle \cite{chung2013} for every sofic approximation, they always have the same sofic pressure. Since the energy is independent of the sofic approximation, so is the sofic entropy.
\end{proof}

\subsection{Some related work}
The paper \cite{burton1995}, written before the introduction of sofic entropy, show that ``Gibbs'' and ``equilibrium'' are not equivalent on trees when entropy is defined by the limit of normalized Shannon entropies of marginals on finite subtrees. They show that for some interactions there is a unique equilibrium state, which is not Gibbs, and for the Ising model with zero external field the free boundary state is the unique equilibrium measure. This is very different from what happens with sofic entropy.

The earlier paper \cite{follmer1977} also noted that Gibbs states on trees can fail to minimize relative entropy density, and introduced an ``inner'' relative entropy density for which the variational principle does hold.

The question is worth revisiting in the context of sofic entropy: as the authors of \cite{burton1995} point out, the source of some of the problems with trees is their nonamenability. Sofic entropy solves some of these difficulties: most famously, it is a measure conjugacy invariant \cite{bowen2010a, bowen2010b}. In the context of statistical physics, sofic pressure is nondecreasing under Glauber dynamics \cite{shriver2022a} and, as mentioned above, it is generally true with sofic entropy that equilibrium measures are Gibbs \cite{barbieri2022a}.

\subsubsection{Related results on the Potts model}
Recent independent work of Coja-Oghlan et al. \cite{coja-oghlan2023} has established analogous behavior for the Potts model on random regular graphs, but the terminology is different. In this section we compare the present work to \cite{coja-oghlan2023}.

The result closest to the present paper is ``phase coexistence'': a typical large random regular graph has some labelings which correspond to the free-boundary state (which they call paramagnetic) but exponentially fewer than the number which correspond to constant-boundary-conditions states (which they call ferromagnetic). The respective exponential growth rates are Bethe free energies of the states -- in this case, this is the same as the $\f$ invariant. Their correspondence between labelings of a finite graph and infinite-volume measures is slightly different form ours: they only require the labelings to use individual colors with the correct frequencies, while we require the average local statistics to match. However, in the temperature range of interest (non-reconstruction/tail-triviality), most of the labelings with the correct color frequencies should be good models for the correct Markov chains, so this should not make a difference. This is because, as discussed below, Markov chains maximize the $\f$ invariant and tail-triviality implies that the $\f$ invariant is the typical entropy value (for Gibbs states of nearest-neighbor interactions -- see Proposition~\ref{prop:tailtrivialf}).

There is also an interesting connection with the main metastability result of \cite{coja-oghlan2023}. Their phase coexistence result is proven by using ``strong non-reconstruction'' to avoid involved second-moment optimization. They then use that the paramagnetic state is a local \emph{maximizer} of the Bethe free energy to argue that there is an exponentially small bottleneck keeping Glauber dynamics near paramagnetic microstates for an exponentially long amount of time. We also use non-reconstruction via Proposition~\ref{prop:tailtrivialf} to avoid doing a second-moment optimization by hand. This result is based on an earlier approach from \cite{shriver2021} that used a weaker but more general metastability result for Glauber evolutions of Gibbs microstates. This does not require the Gibbs states to be local maximizers of the pressure, but it only shows that Glauber dynamics is trapped for an amount of time which grows linearly in the size of the graph (or, in continuous-time Glauber dynamics, for a constant amount of time). The metastability result of \cite{shriver2021} is also more general in that it applies to arbitrary locally tree-like graphs, while \cite{coja-oghlan2023} relies on a counting estimate which holds for typical random graphs.

Note that the Ising free boundary state, which we study in the present paper, is actually a local \emph{minimum} of the $\f$-pressure within the parametrization $\mu_t$ (Proposition~\ref{prop:fnoneq}) so the bottleneck metastability approach of \cite{coja-oghlan2023} will not work. Moreover, the coupling argument for ``strong non-reconstruction'' \cite[Proposition 2.6]{coja-oghlan2023} fails for the Ising model outside the uniqueness regime, so the ``phase coexistence'' result of the present paper (Theorem~\ref{thm:existence}) does not simply follow from \cite{coja-oghlan2023} by taking $q=2$.

See Figures~\ref{fig:pottspressure}, \ref{fig:pottspressure2} for plots of the $\f$-pressures of the families of Markov chains which arise for the Potts model.

We also emphasize that neither the present paper nor \cite{coja-oghlan2023} seems able to completely address phase coexistence when the free-boundary state is reconstructible. Our Theorem~\ref{thm:noneq} states that, for all temperatures below the reconstruction threshold, the set of free-boundary microstates is given exponentially small weight by the Boltzmann measure on any locally tree-like graph. But we do not know whether this weight is nonzero for a typical random regular graph. At low enough temperatures it must be zero (since free-boundary microstates would give bisections with very small cut density, which don't typically exist \cite{dembo2017}), but it seems to be unknown what happens just below the reconstruction threshold.

\begin{figure}
\begin{center}
	\includegraphics{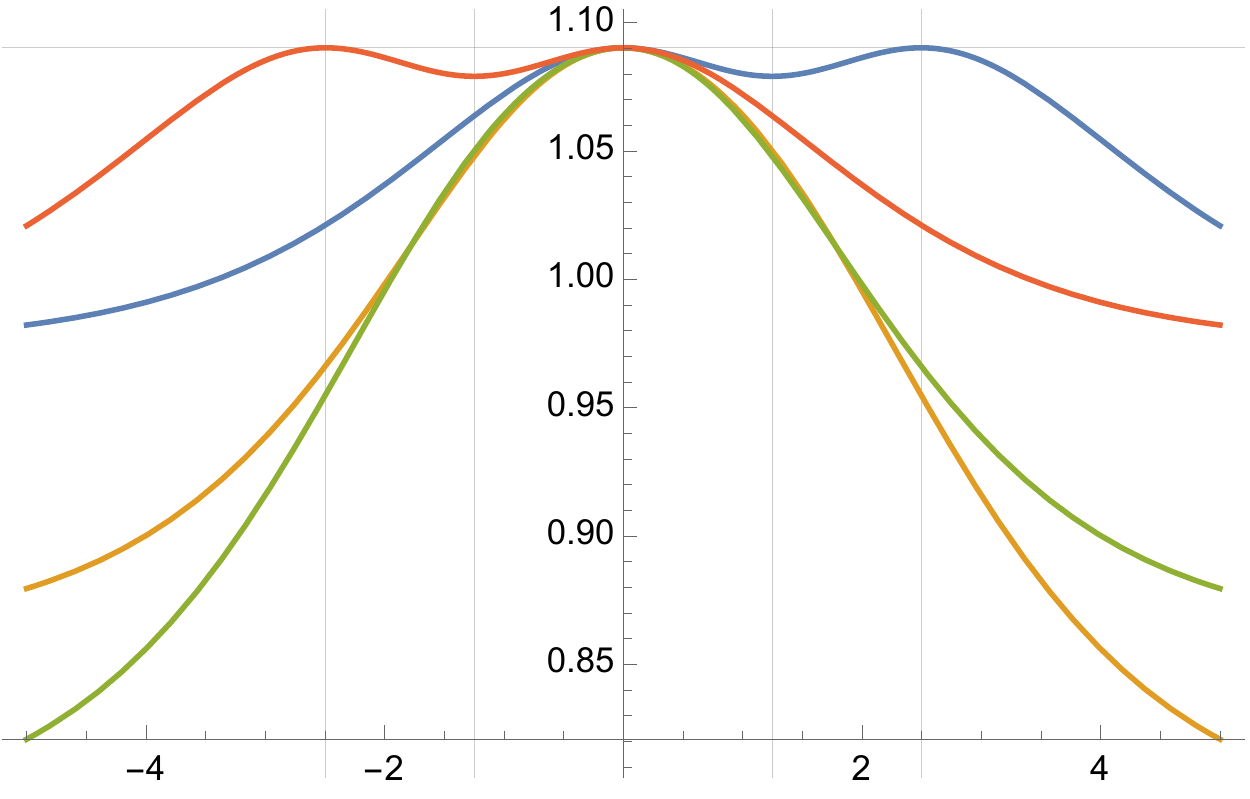}
\end{center}
\caption{For $q=5$, ignoring most permutations of the colors, there are $q-1=4$ families of Markov chains arising from the belief propagation equations, analogous to the single family $\mu_t$ defined in Section~\ref{sec:isingintro} below. For every family, $t=0$ corresponds to the free boundary (paramagnetic) state. The plot shows the $\f$-pressures of these four families at the value of $J$ where the free boundary state is overtaken by the constant boundary conditions (ferromagnetic/ordered) states. Note that there are also local minima of the pressure here: these correspond to more Markov chains which are Gibbs states, and hence have metastable microstates in the sense of \cite{shriver2021}.}
\label{fig:pottspressure}
\end{figure}
\begin{figure}
\begin{center}
	\includegraphics{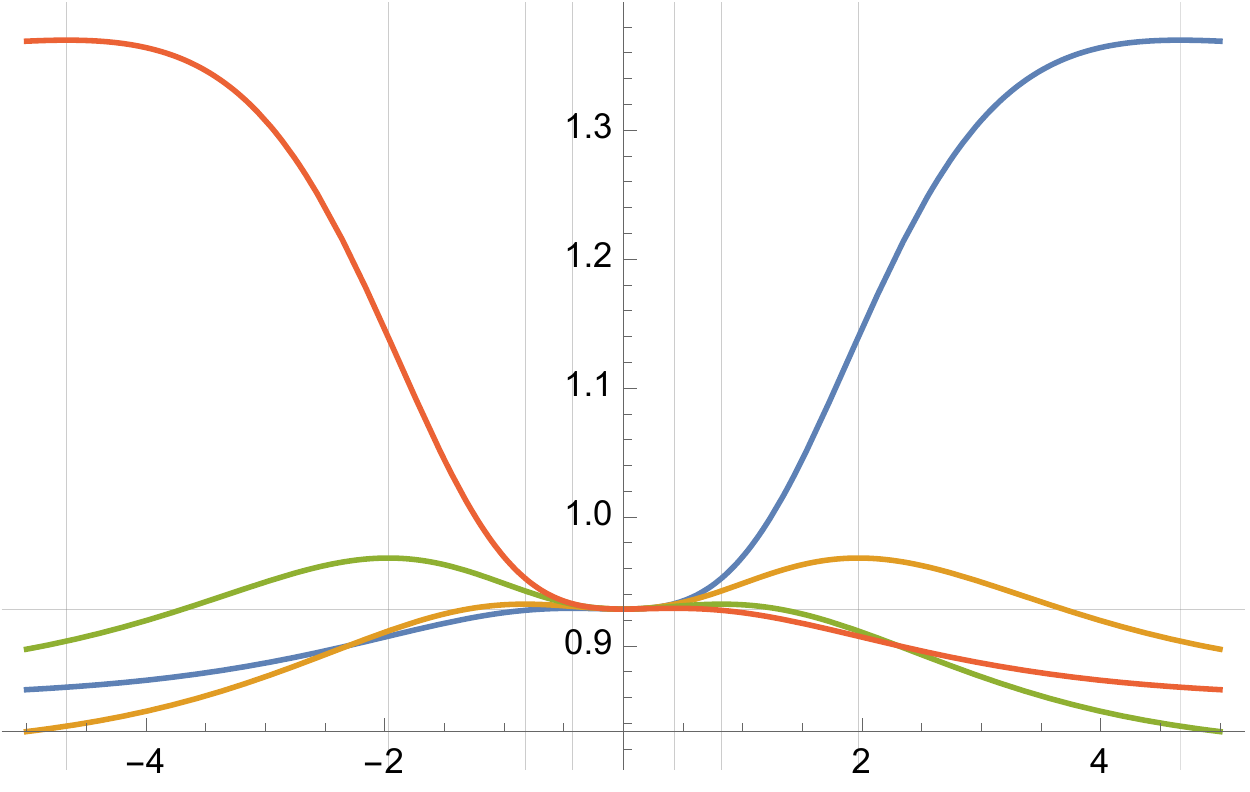}
\end{center}
\caption{At higher values of $J$, the other families contain additional Markov Gibbs states. This behavior is investigated further in \cite{kulske2014}. In the plot, these states are marked by vertical lines and coincide with critical points of the pressure.}
\label{fig:pottspressure2}
\end{figure}

\subsection{Acknowledgements}
This material is based on work completed at the University of Texas at Austin while supported by NSF Grant DMS 1937215.

The author thanks Tim Austin, Sebasti\'{a}n Barbieri, Lewis Bowen, Raimundo Brice\~{n}o, and Lorenzo Sadun for helpful conversations. Thanks to Tim Austin and Lewis Bowen for many helpful comments on earlier drafts.

Plots were produced using gnuplot and Mathematica.

\section{Definitions}
\label{sec:definitions}
Although the main results are stated for the Ising model on a tree, some results below (for example, Theorem~\ref{thm:eqgeneral}) are proven in greater generality. Let $\Gamma$ be a countable sofic group and $\A$ be a finite set.


\subsection{Sofic entropy and sofic groups}
We give a brief summary here to establish notation. The concept of sofic entropy is originally due to Lewis Bowen \cite{bowen2010b} but we follow the notation of \cite{austin2016} for shift systems.

If $V$ is a finite set, given $\mb{x} \in \A^V$, $v \in V$, $\sigma \colon \Gamma \to \Sym(V)$ we define the \emph{pullback name} $\Pi_v^\sigma \mb{x} \in \A^\Gamma$ by
	\[ \Pi_v^\sigma \mb{x} (\gamma) = \mb{x}\big( \sigma(\gamma) (v) \big) . \]
The \emph{empirical distribution} of $\mb{x}$ over $\sigma$ is
	\[ P_{\mb{x}}^\sigma = \big( v \mapsto \Pi_v^\sigma \mb{x} \big)_* \Unif(V) = \frac{1}{\abs{V}} \sum_{v \in V} \delta_{\Pi_v^\sigma \mb{x}} . \]
If $\calO$ is a weak*-open subset of $\Prob(\A^\Gamma)$, we define
	\[ \Omega(\sigma, \calO) = \{ \mb{x} \in \A^V \st P_{\mb{x}}^\sigma \in \calO \} . \]
If $\calO$ is a small neighborhood of some $\mu$ then we can think of any $\mb{x} \in \Omega(\sigma, \calO)$ as a ``good model'' for $\mu$ over $\sigma$ in that its ``local statistics'' are close to $\mu$.

If $\Sigma = \big( \sigma_n \colon \Gamma \to \Sym(V_n) \big)_{n \in \NN}$ is a sequence of maps then the sofic entropy of $\mu \in \Prob(\A^\Gamma)$ over $\Sigma$ is defined by
	\[ \h_\Sigma(\mu) = \inf_{\calO \ni \mu} \limsup_{n \to \infty} \frac{1}{\abs{V_n}} \log \abs*{\Omega(\sigma_n, \calO)} . \]
This is the (upper) exponential growth rate of the number of good models for $\mu$ over $\Sigma$.

Given $F \Subset \Gamma$ and $\delta > 0$, a map $\sigma \colon \Gamma \to \Sym(V)$ is called $(F,\delta)$-\emph{sofic} if
	\[ \frac{1}{\abs{V}} \abs*{\big\{ v \in V \st \sigma(\gamma)(v) \ne v \ \forall \gamma \in F \setminus\{1_\Gamma\} \big\}}  > 1-\delta \quad \text{(almost free)}\]
and
	\[ \frac{1}{\abs{V}} \abs*{\big\{ v \in V \st \sigma(\gamma_1 \gamma_2)(v) = \sigma(\gamma_1)\big(\sigma(\gamma_2)(v)\big) \ \forall\gamma_1, \gamma_2 \in F \big\}}  > 1-\delta \quad \text{(almost action)}. \]
A sequence $\Sigma = \{\sigma_n \colon \Gamma \to \Sym(V_n)\}_{n=1}^\infty$ is a sofic approximation (to $\Gamma$) if $\abs{V_n} \to \infty$ and for every $F,\delta$, the map $\sigma_n$ is $(F,\delta)$-sofic for all large enough $n$.

A countable group $\Gamma$ is called sofic if there is at least one sofic approximation to $\Gamma$.

If $\Sigma$ is a sofic approximation then $\h_\Sigma(\cdot)$ is a measure-conjugacy invariant \cite{bowen2010b}, and $\h_\Sigma(\mu) = -\infty$ if $\mu$ is not invariant under the shift action of $\Gamma$.

\subsection{Random sofic approximations, $\f$-entropy}
Suppose we have a sequence $\Sigma = (\sigma_n)_{n \in \NN}$ of \emph{random} maps $\sigma_n \colon \Gamma \to \Sym(V_n)$. Then $\Sigma$ is called a (random) sofic approximation if for any $(F,\delta)$ and any $c>0$ we have 
	\[ \PP( \sigma_n \text{ is } (F,\delta) \text{-sofic}) > 1 - e^{-cn} \]
for all large enough $n$. The sofic entropy of $\mu$ over a random $\Sigma$ is
	\[ \h_\Sigma(\mu) = \inf_{\calO \ni \mu} \limsup_{n \to \infty} \frac{1}{\abs{V_n}} \log \EE \abs*{\Omega(\sigma_n, \calO)} . \]
The expectation is taken over the random $\sigma_n$. 

In the case of a free group $\Gamma = \FF_r = \langle s_1, \ldots, s_r \rangle$, one natural choice is to let $\sigma_n$ be a uniformly random \emph{homomorphism} in $\Hom(\Gamma,\Sym(n))$. We call this sequence, which is a random sofic approximation, $\f$. In \cite{bowen2010} it is shown that $\h_{\f}(\mu)$ is equal to the ``$\f$ invariant'' $\f(\mu)$, a measure-conjugacy invariant for free-group actions introduced in \cite{bowen2010a}. For general systems, it should be noted that the two are equal only when there is a finite-entropy generating partition, but this is satisfied for finite-alphabet shift systems. More recently, $\f(\mu)$ has sometimes been called ``RS entropy'' (where RS stands for ``replica-symmetric'') or ``annealed entropy'' to emphasize analogy with related ideas in statistical physics \cite{bowen2019,backhausz2022}.

Still in the case of a free group, if $\mu \in \Prob^\Gamma(\A^\Gamma)$ is a Markov chain then there is an easy-to-compute formula for $\f(\mu)$. 
For $\gamma \in \Gamma$ let $\pi_\gamma \colon \A^\Gamma \to \A$ denote projection onto the $\gamma$ coordinate. Then, as shown in \cite{bowen2010c},
	\[ \f(\mu) = (1-r) \shent_\mu (\pi_e) + \sum_{i = 1}^r \shent_\mu (\pi_e \mid \pi_{s_i}) .  \]

\subsection{Equilibrium measures and pressure}

We sometimes want to work with measures supported on some closed, shift-invariant subset $\calX \subset \A^\Gamma$, called a \emph{subshift}. For example, independent sets can be considered as labelings of $\Gamma$ by elements of $\{0,1\}$ such that no two $1$'s are adjacent (in the Cayley graph for some fixed generating set). 

The same approach to sofic pressure on a subshift appears in \cite[Section 3.4]{barbieri2022a}.

Let $u \colon \calX \to \RR$ be a continuous function. This will play the role of a ``specific energy'' or ``energy per site'' functional without having to perform a spatial average (which is not available for nonamenable groups).

Since $\A^\Gamma$, and hence $\calX$, is compact, $u$ has minimum and maximum values $u^{\min}$ and $u^{\max}$.

For $\mu \in \Prob^\Gamma(\calX)$, we define the \emph{energy density} of $\mu$ by
	\[ u(\mu) = \int_{\calX} u(x)\, \mu(dx) . \]
	
The \emph{pressure} of $\mu$ over some (random or deterministic) sofic approximation $\Sigma$ is then defined by
	\[ \press_\Sigma(\mu) = \h_\Sigma(\mu) - u(\mu). \]
	
Any $\mu$ maximizing the pressure over $\Sigma$ is called $\Sigma$-equilibrium. At least one such $\mu$ always exists by semicontinuity, but in general there can be many.

We will also need to work with pressure for measures on finite ``model spaces'' $\A^{V}$, but this requires dealing with a small issue: if $\sigma \colon \Gamma \to \Sym(V)$ is any map, then we would want to define the total energy of a configuration $\mb{x} \in \A^V$ by summing up the energy contributions from each vertex, each of which can be determined by pulling back to a configuration on $\Gamma$:
	\[ U^\sigma(\mb{x}) = \sum_{v \in V} u(\Pi_v^\sigma \mb{x}) = \abs{V} \cdot \int_{\calX} u(x)\, P_{\mb{x}}^\sigma(dx) . \]
In the case of the Ising model, this is fine. But in general we cannot be sure that every pullback name $\Pi_v^\sigma \mb{x}$ is in the subshift $\calX$ (i.e. the domain of $u$). So we need to pick a continuous extension $\tilde{u} \colon \A^\Gamma \to \RR$ and define
	\[ \tilde{U}^\sigma(\mb{x}) = \sum_{v \in V} \tilde{u}(\Pi_v^\sigma \mb{x}) = \abs{V} \cdot \int_{\A^\Gamma} \tilde{u}(x)\, P_{\mb{x}}^\sigma(dx) . \]
If $\mb{x}$ is a good model for some $\mu \in \Prob^\Gamma(\calX)$, then $\frac{1}{\abs{V}}\tilde{u}(\mb{x}) \approx \tilde{u}(\mu) = u(\mu)$.

For $\zeta \in \Prob(\A^V)$, we then define the average energy
	\[ \tilde{U}^\sigma(\zeta) = \int_{\A^V} \tilde{U}^\sigma(\mb{x}) \, \zeta(d\mb{x}) = \abs{V} \cdot \int_{\A^\Gamma} \tilde{u}(\mb{x})\, P_\zeta^\sigma(d\mb{x}) \]
and the pressure
	\[ \widetilde{\Press}^\sigma(\zeta) = \shent(\zeta) -  \tilde{U}^\sigma(\zeta). \]
It is straightforward to show using Lagrange multipliers that the unique $\zeta$ maximizing $\widetilde{\Press}^\sigma(\cdot)$ is the measure $\tilde{\xi}_\sigma$ with $\tilde{\xi}_\sigma\{\mb{x}\} = \tilde{Z}_\sigma^{-1} \exp(-\tilde{U}^\sigma(\mb{x}))$, where the normalizing constant $\tilde{Z}_\sigma = \sum_{\mb{x}} \exp(-\tilde{U}^\sigma(\mb{x}))$ is called the \emph{partition function}. It is also straightforward to check that $\widetilde{\Press}^\sigma(\xi_\sigma) = \log \tilde{Z}_\sigma$.

But this $\tilde{\xi}_\sigma$ may depend on the choice of extension $\tilde{u}$, and it may not be supported on good models for measures on $\calX$. To fix this, given an open neighborhood $\calU \supset \Prob^\Gamma(\calX)$ let $\tilde{\xi}_\sigma^{\calU}$ be the measure supported on $\Omega(\sigma, \calU)$ with maximal $\tilde{P}^\sigma$. Again using Lagrange multipliers, we get that
	\[ \tilde{\xi}_\sigma^\calU \{\mb{x}\} = \tilde{Z}_{\sigma,\calU}^{-1} \exp(-\tilde{U}^\sigma(\mb{x})) \1_{\mb{x} \in \Omega(\sigma, \calU)} \]
where
	\[ \tilde{Z}_{\sigma,\calU} = \sum_{ \mb{x} \in \Omega(\sigma, \calU)} \exp(-\tilde{U}^\sigma(\mb{x})) . \]

The \emph{topological pressure} in this setting can be defined by
	\[ \press_\Sigma = \inf_{\calU \supset \Prob^\Gamma(\calX)} \limsup_{n \to \infty} \frac{1}{n} \log \EE_{\sigma_n} \tilde{Z}_{\sigma_n,\calU} . \]
The infimum over neighborhoods $\calU$ of $\Prob^\Gamma(\calX)$ eliminates the dependence on the extension $\tilde{u}$.

The sofic, deterministic version of this quantity was introduced in \cite{chung2013}, where it was also shown that, for deterministic $\Sigma$,
	\[ \press_{\Sigma} = \sup \{ \press_\Sigma(\mu) \st \mu \in \Prob^\Gamma(\calX) \} . \]
In other words, $\press_\Sigma$ is the value of ${\press}_\Sigma(\mu)$ for every $\Sigma$-equilibrium $\mu$.

\subsection{The Ising Model}
\label{sec:isingintro}
In this paper we mostly work with the ferromagnetic Ising model. For a free group $\FF_r$ with generating set $\langle s_1, \ldots, s_r \rangle$ the Ising model can be defined by setting
	\[ u(\mb{x}) = -\frac{J}{2} \sum_{i=1}^r \big( \mb{x}(e) \mb{x}(s_i) + \mb{x}(e) \mb{x}(s_i^{-1}) \big) \]
for $\mb{x} \in \{-1, +1\}^{\FF_r}$, where $J>0$ is some fixed ``interaction strength.'' We can think of each edge with the same label on either end as having energy $-J$ and each edge with opposite labels as having energy $+J$: we then assign half the energy of each edge to each of the vertices it's incident on. Then the energy assigned to the identity $e$ is given by the formula above.

Often a parameter called ``inverse temperature'' $\beta$ is included in addition to or instead of $J$ (usually not in the definition of $u$, but rather as a multiplicative factor often appearing along with $u$), but mathematically this is redundant in this setting. Here we follow the notation of \cite{georgii2011}. But we do use the standard terminology of ``low temperature'' (equivalent to large $J$) and ``high temperature'' (small $J$).

Depending on $J$ and $r$, there are up to three completely homogeneous Markov chains which are Gibbs measures for the Ising model. Following \cite[Section 12.2]{georgii2011}, for each $t \in \RR$, let
	\[ \alpha(t)=\frac{e^{-2J}+e^{-2t}}{2e^{-2J}+2\cosh(2t)} \]
and
	\[ \beta(t)=\frac{1}{e^{2(J+t)}+1} .\]
Then let $\mu_t$ be the completely homogeneous (i.e. invariant under the full automorphism group of the Cayley graph) Markov chain on $\{-1,+1\}^\Gamma$ with single-site marginal the row vector
	\[ \left(\begin{array}{cc}
		\PP\{-1\} & \PP\{+1\}
	\end{array} \right)
	=
	\left(\begin{array}{cc}
		\alpha(t) & 1-\alpha(t)
	\end{array} \right) \]
and transition matrix
	\[ \left(\begin{array}{cc}
		\PP\{-1 \to -1\} & \PP\{-1 \to +1\} \\
		\PP\{+1 \to -1\} & \PP\{+1 \to +1\}
	\end{array} \right)
	=
	\left(\begin{array}{cc}
		1-\beta(-t) & \beta(-t) \\
		\beta(t) & 1-\beta(t)
	\end{array} \right) . \]
In \cite[Section 12.2]{georgii2011} it is shown that the Gibbs measures for the Ising model on $\Gamma = \FF_r$ which are completely homogeneous Markov chains are exactly the measures $\mu_t$ where $t$ is a solution to 
	\[ t = \frac{2r-1}{2} \log \frac{\cosh(t+J)}{\cosh(t-J)} . \]
(Note that $d$ in \cite{georgii2011} denotes the branching factor, not the degree.)

For any $J,r$ we have a solution $t=0$, which is called the \emph{free boundary} Gibbs state $\mu^{\FB}$. For $J > J_{uniq}(r) \coloneqq \arctanh(\frac{1}{2r-1})$, there is a pair of solutions $t_+, t_-$ which give two additional states called the plus and minus boundary condition states, respectively. In fact $t_+ = - t_-$. We write $\mu^+ = \mu_{t^+}$ for the plus state and $\mu^- = \mu_{t^-}$ for the minus state. If $J \leq J_{uniq}(r)$ then $t=0$ is the only solution, and in fact $\mu_0 = \mu^{\FB}$ is the only Gibbs measure. For that reason $J_{uniq}(r)$ is called the uniqueness threshold.

There is a second threshold called the reconstruction threshold, $J_{rec}(r) = \arctanh \frac{1}{\sqrt{2r - 1}}$. The free boundary state is tail-trivial if and only if $0 < J \leq J_{rec}(r)$ \cite{bleher1995}. Tail-triviality is equivalent to being an extreme point in the convex set of all (not necessarily shift-invariant) Gibbs measures. Many authors also use the term ``pure.''

\section{The free boundary state is $\f$-nonequilibrium}
In this section we prove Proposition \ref{prop:fnoneq}. Throughout, the external field is assumed to be zero.

We prove Proposition \ref{prop:fnoneq} by showing that $t=0$ is a local minimum of $t \mapsto \press_{\f}(\mu_t)$ (where $\mu_t$ is as defined in Section \ref{sec:isingintro}). There may be other parametrizations which work for this purpose.


\subsection{Specific energy}
In this subsection we find an expression for the specific energy $u(\mu_t)$.

The energy of the configuration $\mbf{x} \in \{\pm 1\}^\Gamma$ due to interactions with the identity $e \in \Gamma$ is 
	\[ u(\mbf{x}) = \frac{1}{2} \sum_{v \text{ adj. to } e} -J \mbf{x}(e) \mbf{x}(v) = -\frac{J}{2} \sum_{i = 1}^r [ \mbf{x}(e) \mbf{x}(s_i) + \mbf{x}(e) \mbf{x}(s_i^{-1})] . \]
The factor of $\frac{1}{2}$ appears because each edge is involved in the interaction of two particles, so we assign half the energy of the edge to each. The corresponding definition in \cite{georgii2011} is Equation (15.22).

Each of the $2r$ products of adjacent spins has the same distribution under $\mu_t$, and
\begin{align*}
	\EE_{\mbf{x} \sim \mu_t} \mbf{x}(e) \mbf{x}(s_i)
		&= \PP\{\mbf{x}(e) = \mbf{x}(s_i)\} - \PP\{ \mbf{x}(e) \ne \mbf{x}(s_i)\} \\
		&= \left[ \alpha(t) (1-\beta(-t)) + (1-\alpha(t))(1-\beta(t)) \right] - \left[ \alpha(t) \beta(-t) + (1-\alpha(t)) \beta(t) \right] \\
		&= \left[ \alpha(t) -\alpha(t)\beta(-t) + 1 - \beta(t) - \alpha(t) + \alpha(t)\beta(t) \right] - \left[ \alpha(t) \beta(-t) + \beta(t)-\alpha(t)\beta(t) \right] \\
		&= 1 - 2\alpha(t) \beta(-t) - 2\beta(t) + 2 \alpha(t)\beta(t) \\
		&= 1 + 2\alpha(t)[\beta(t) - \beta(-t)] - 2\beta(t).
\end{align*}
So the specific energy is
	\[ u(\mu_t) \coloneqq \EE_{\mbf{x} \sim \mu_t} u(\mbf{x}) = -Jr \left( 1 + 2\alpha(t)[\beta(t) - \beta(-t)] - 2\beta(t) \right) . \]

\subsection{$\f$ invariant}
In this subsection we find an expression for $\f(\mu_t)$.

Since $\mu_t$ is a Markov chain, $\f(\mu_t) = F(\mu_t)$ \cite{bowen2010c}.
\begin{align*}
	F(\mu_t)
		&= (1-r) \shent_{\mbf{x} \sim \mu_t} (\mbf{x}(e)) + \sum_{i=1}^r \shent_{\mbf{x} \sim \mu_t} (\mbf{x}(s_i) | \mbf{x}(e)) \\
		&= (1-r) \shent(\alpha(t)) + r \left[ \alpha(t) \shent(\beta(-t)) + (1-\alpha(t)) \shent(\beta(t))\right]
\end{align*}
where in the second line we use $\shent$ to denote the binary entropy function $\shent(p) = -p \log p - (1-p) \log (1-p)$.

\subsection{Proof of Proposition \ref{prop:fnoneq}}
We have shown so far that
\begin{align*}
	\press_{\f}(\mu_t)
		&=  \f(\mu_t) - u(\mu_t) \\
		&= Jr \left( 1 + 2\alpha(t)[\beta(t) - \beta(-t)] + 2\beta(t) \right) \\
		&\hspace{1in} + (1-r) \shent(\alpha(t)) - r \left[ \alpha(t) \shent(\beta(-t)) + (1-\alpha(t)) \shent(\beta(t))\right] .
\end{align*}
The derivative is quite complicated, but using Mathematica \cite{Mathematica}, for example, we can then calculate $\left.\frac{d}{dt}\right|_{t=0} \press_{\f}(\mu_t) = 0$ and
	\[ \left.\frac{d^2}{dt^2}\right|_{t=0} \press_{\f}(\mu_t) = (\tanh (J)+1) \left[ (2r-1) \tanh (J)-1 \right] . \]
Since $\tanh J + 1 > 0$ for all $J$, this is positive exactly when
	\[ J > \arctanh[(2r-1)^{-1}] . \]
So for all temperatures below the uniqueness threshold, the free boundary state $\mu_0$ is a local minimum of $t \mapsto \press_{\f}(\mu_t)$, so cannot be $\f$-equilibrium.


%
%

\section{Pressure over a typical sofic approximation (Proof of Theorem \ref{thm:existence})}
Suppose that two Gibbs states $\mu, \nu$ are both tail-trivial but have different $\f$-pressures. For the Ising model, by \cite[Corollary 16]{shriver2021} this implies that $\mu,\nu$ both satisfy a standard kind of ``second moment'' criterion. For $\mu$, for example, this means that if $\lambda$ is any self-joining of $\mu$ (that is, $\lambda$ is a shift-invariant measure on $(\{+1, -1\}^2)^{\FF_r}$ with both $\{+1,-1\}^{\FF_r}$-marginals equal to $\mu$) then $\f(\lambda) \leq 2 \f(\mu)$.

The next proposition shows that this second-moment criterion implies that for a ``typical'' sofic approximation $\Sigma$, the $\Sigma$-entropy of both measures are the same as their $\f$-invariants. Combined with Proposition \ref{prop:fnoneq}, this proves Theorem \ref{thm:existence}, since $\mu^{\FB}$ is tail-trivial at temperatures down to the reconstruction threshold and $\mu^+,\mu^-$ are always tail-trivial.

\begin{prop}
\label{prop:tailtrivialf}
	Let $\Gamma = \mathbb{F}_r$ denote the rank-$r$ free group. If $\mu \in \Prob^\Gamma(\A^\Gamma)$ is such that $\f(\lambda)\leq 2\f(\mu)$ for any self-joining $\lambda$ of $\mu$, there exists a sequence $\{S_n \subset \Hom(\Gamma, \Sym(n)\}_{n=1}^\infty$ of sets of homomorphisms with $\PP(S_n) \to 1$ such that if $\sigma_n \in S_n$ for each $n$ then
	\begin{enumerate}
		\item $\Sigma = \{\sigma_n\}_{n=1}^\infty$ is a sofic approximation and
		\item $\h_\Sigma(\mu) = \f(\mu)$.
	\end{enumerate}
\end{prop}

\subsection{Concentration}
For $h \in [0, \log|\A|]$ and $\sigma \in \Hom(\Gamma, \Sym(n))$, let
	\[ \rho_\mu (\sigma, h) = \inf \left\{ \varepsilon>0 \st \frac{1}{n} \log \abs*{\Omega(\sigma, \mu, \varepsilon)} \geq h \right\} . \]
Note that $\rho_\mu(\sigma, h) \leq \diam(\A^n)$. For fixed $h$ and $\mu$, this quantity exhibits exponential concentration as a function of a random homomorphism $\sigma$:

\begin{lemma}
\label{lem:concentration}
	If $\sigma_n \sim \Unif(\Hom(\Gamma, \Sym(n)))$ then
	\[ \PP \left\{ \abs*{ \rho_\mu (\sigma_n, h) - \EE [\rho_\mu(\sigma_n, h)]} > t \right\} \leq 2 \exp\left( \frac{-n t^2}{8r} \right) . \]
\end{lemma}
\begin{proof}
	 This is a consequence of a standard kind of martingale concentration argument using ``switchings''; a version for the permutation model appears in \cite{shriver2021}. If two homomorphisms $\sigma_1, \sigma_2$ differ by a ``switching'' then for any microstate, its empirical distributions over $\sigma_1, \sigma_2$ are distance at most $2/n$ apart. So the minimum distance from $\mu$ to get any specified number of microstates changes by at most $2/n$. Hence
	 \[ \PP \left\{ \abs*{ \rho_\mu (\sigma_n, h) - \EE [\rho_\mu(\sigma_n, h)]} > t \right\} \leq 2 \exp\left( \frac{-t^2}{2nr(2/n)^2} \right). \qedhere\]
\end{proof}

It is not immediately clear whether $\EE [\rho_\mu(\sigma_n, h)]$ converges as $n \to \infty$, but we can define the upper limit
	\[ \overline\rho_\mu (h) = \limsup_{n \to \infty} \EE [\rho_\mu(\sigma_n, h)]. \]
This is a nondecreasing function of $h$.
\begin{lemma}
\label{lem:rhozero}
	If the product self-joining $\mu \times\mu$ has maximal $\f$ invariant, then $\overline\rho_\mu(h) = 0$ for every $h < \f(\mu)$.
\end{lemma}
\begin{proof}
	By the Paley-Zygmund inequality, for any $\theta \in (0,1)$
		\[ \PP\left\{ \abs*{\Omega(\sigma_n, \mu, \varepsilon)} \geq \theta \EE\abs*{\Omega(\sigma_n, \mu, \varepsilon)} \right\} \geq (1-\theta)^2 \frac{\big( \EE\abs*{\Omega(\sigma_n, \mu, \varepsilon)} \big)^2}{\EE\! \big[ \abs*{\Omega(\sigma_n, \mu, \varepsilon)}^2 \big]} . \]
	Now by a standard compactness argument
		\[ \inf_{\varepsilon>0} \limsup_{n \to \infty} \frac{1}{n} \log \EE\! \big[ \abs*{\Omega(\sigma_n, \mu, \varepsilon)}^2 \big] = \sup\{ \f(\lambda) \st \lambda \text{ self-joining of $\mu$} \} ,  \]
	which by assumption is equal to $\f(\mu \times \mu) = 2 \f(\mu)$. So for any $\delta>0$, if $\varepsilon>0$ is small enough then for all large enough $n$
		\[ \EE\! \big[ \abs*{\Omega(\sigma_n, \mu, \varepsilon)}^2 \big] \leq \exp\big( n ( 2 \f(\mu) + \delta) \big) .\]
	It is a consequence of \cite{bowen2010c} that
		\[ \f(\mu) = \inf_{\varepsilon>0} \liminf_{n \to \infty} \frac{1}{n} \log \EE \abs*{\Omega(\sigma_n, \mu, \varepsilon)}; \]
	whether a limsup or liminf is used, the method of \cite{bowen2010c} shows that this coincides with the earlier information-theoretic formula for $\f(\mu)$ in \cite{bowen2010}. In particular, for any $\delta>0$, for all $\varepsilon>0$ we have
		\[ \EE \abs*{\Omega(\sigma_n, \mu, \varepsilon)} \geq \exp \big( n( \f(\mu) - \delta) \big) \]
	for all large enough $n$.
	
	So given $\delta>0$, if $\varepsilon$ is small enough then for all large $n$
		\[ \PP\left\{ \abs*{\Omega(\sigma_n, \mu, \varepsilon)} \geq \theta \EE\abs*{\Omega(\sigma_n, \mu, \varepsilon)} \right\} \geq (1-\theta)^2 \frac{\big(\exp \big( n( \f(\mu) - \delta) \big) \big)^2}{ \exp\big( n ( 2 \f(\mu) + \delta) \big)} = (1-\theta)^2 \exp(-3\delta n). \]
		
	Now fix some $h < \f(\mu)$. Pick some $\delta>0$ such that $\f(\mu)-\delta > h$; then for any $\varepsilon>0$ and $\theta \in (0,1)$ for all large enough $n$ we have
		\[ \theta \EE\abs*{\Omega(\sigma_n, \mu, \varepsilon)} \geq \theta \exp(n (\f(\mu) - \delta)) \geq \exp(n h). \]
	So for any $\delta>0$, if $\varepsilon>0$ is small enough then for all large enough $n$
		\[ \PP\{ \rho_\mu(\sigma_n, h) \leq \varepsilon\} \geq \PP\{ \tfrac{1}{n} \log \abs*{\Omega(\sigma_n, \mu, \varepsilon)} \geq h \} \geq (1-\theta)^2 \exp(-3\delta n). \]
	Since $\varepsilon,\delta>0$ can be as small as we want, this is incompatible with the exponential concentration in Lemma~\ref{lem:concentration} unless $\overline\rho_\mu(h) = 0$.
\end{proof}

\subsection{Proof of Proposition~\ref{prop:tailtrivialf}}
Part of the proof is based on the proof of \cite[Theorem 1.1]{airey2022}.

	Let $\gamma_n \sim \Unif(\Hom(\Gamma,\Sym(n))$.
	For each $n \in \NN$ and $\varepsilon,\delta>0$ we define the following two sets of ``bad'' homomorphisms: if $\f(\mu)$ is finite, then
	\begin{align*}
		\mathcal{F}^{\geq}(\varepsilon,\delta,n)
			&= \left\{ \frac{1}{n} \log \abs*{\Omega(\gamma_n, \mu, \varepsilon)} \geq \f(\mu) + \delta \right\} \\
		\mathcal{F}^{\leq}(\varepsilon,\delta,n)
			&= \left\{ \frac{1}{n} \log \abs*{\Omega(\gamma_n, \mu, \varepsilon)} \leq \f(\mu) - \delta \right\}
	\intertext{If $\f(\mu) < 0$ then we set $\mathcal{F}^{\leq}(\varepsilon,\delta,n) = \varnothing$ and in the definition of $\mathcal{F}^{\geq}(\varepsilon,\delta,n)$ we change $\f(\mu) + \delta$ to $0$. For a finite subset $F \subset \Gamma$ we define}
		\mathcal{F}^3(F,\delta, n)
			&= \left\{ \gamma_n \text{ is not } (F,\delta)\text{-sofic} \right\}
	\end{align*}
	All of these sets have small probability:
	\begin{enumerate}
		\item For each $\delta > 0$ and small enough (depending on $\delta$) $\varepsilon > 0$ there exists $\xi>0$ such that for all large $n$
			\[ \PP \left( \mathcal{F}^{\geq}(\varepsilon,\delta,n) \right) < e^{-\xi n} . \]
		This is a consequence of Markov's inequality.
		
		\item For each $\varepsilon, \delta > 0$ there exists $\xi>0$ such that for all large $n$
			\[ \PP \left( \mathcal{F}^{\leq}(\varepsilon,\delta,n) \right) \leq e^{-\xi n} . \]
		\emph{Proof}. If $\f(\mu) < 0$ then this is immediate since $\mathcal{F}^{\leq} = \varnothing$. So assume $\f(\mu) \geq 0$. If the event $\mathcal{F}^{\leq}(\varepsilon,\delta,n)$ occurs then $\rho_\mu(\sigma_n, \f(\mu)-\delta/2) \geq \varepsilon$. But by Lemma~\ref{lem:rhozero} $\overline\rho_\mu(\f(\mu)-\delta/2)=0$, so by concentration of $\rho_\mu(\sigma_n, h)$ (Lemma~\ref{lem:concentration}) this is exponentially unlikely to occur for all large $n$. \hfill \qedsymbol
		
		\item For each finite $F \subset \Gamma$ and $\delta,c > 0$, for all large enough $n$
			\[ \PP \left( \mathcal{F}^3(F,\delta, n) \right) < e^{-cn} . \]
		This is (a weaker version of) \cite[Lemma 3.1]{airey2022} with minor modifications.
	\end{enumerate}
	
	We now use these bounds to construct the sequence $(S_n)$.
	
	Pick a sequence $(\delta_m)_{m \in \NN}$ decreasing with limit 0, then pick a decreasing sequence $(\varepsilon_m)_{m \in \NN}$ small enough that the bound in item 1 above holds.  For $\diamond \in \{ \leq, \geq\}$ let $\calE^{\diamond}_{m,n} = \mathcal{F}^{\diamond}(\varepsilon_m, \delta_m, n)$.
	
	For each $m$ let $\xi_m>0$ be such that $\PP(\calE^{\diamond}_{m,n}) < e^{-\xi_m n}$ for all large enough $n$ and all $\diamond \in \{\leq,\geq\}$. Pick an exhaustion $(F_m)_{m \in \NN}$ of $\Gamma$ and let $\calE^3_{m,n} = \mathcal{F}^3(F_m, \varepsilon_m, n)$.
	
	Choose an increasing sequence $K_m$ so that for all $n > K_m$
	\begin{align*}
		3 \sum_{i=1}^m e^{-\xi_i n}
			&\leq \tfrac{1}{m} , \\
		\PP \left[ \calE^{\diamond}_{i,n} \right]
			&\leq e^{-\xi_i n} \quad \text{ for all } i \in [m],\ \diamond \in \{\leq, \geq\} \\
		\PP \left[ \calE^3_{i,n} \right]
			&\leq e^{-\xi_i n} \quad \text{ for all } i \in [m] .
	\end{align*}
	Now let $\calU_{m,n} = \left( \bigcup_{i=1}^m (\calE^{\geq}_{i,n} \cup\calE^{\leq}_{i,n} \cup \calE^3_{i,n}) \right)^c$. For all $n > K_m$ we have
		\[ \PP \left[ \calU_{m,n} \right] \geq  1 - 3 \sum_{i=1}^m e^{-\xi_i n}\geq 1 - \tfrac{1}{m} . \]
		
	Now for each $m$ and each $n \in (K_m, K_{m+1}]$ let $S_n = \calU_{m,n}$. Note that for each $m,n$ we have $\calU_{m,n} \supset \calU_{m+1,n}$. So if $\Sigma = \{\sigma_n\}_{n = 1}^\infty$ is such that $\sigma_n \in S_n$ for each $n$, this ensures that for each $m$, for all $n > K_m$ we have $\sigma_n \in \calU_{m,n}$. In particular, 
	\begin{enumerate}
		\item For each $m$, for all $n > K_m$ we have $\sigma_n \not\in \calE^{\geq}_{m,n}$. So for all large $n$ we have $\sigma_n \not\in \calF^{\geq}(\varepsilon_m, \delta_m, n)$. If $\f(\mu) \geq 0$ this means
			\[ \limsup_{n \to \infty} \frac{1}{n} \log \abs*{\Omega(\sigma_n, \mu, \varepsilon_m)} \leq \f(\mu) + \delta_m . \]
		If $\f(\mu) = -\infty$ it means $\log \abs*{\Omega(\sigma_n, \mu, \varepsilon_m)} = -\infty$ for all large enough $n$, so
			\[ \limsup_{n \to \infty} \frac{1}{n} \log \abs*{\Omega(\sigma_n, \mu, \varepsilon_m)} = -\infty . \]
		In either case, taking $m \to \infty$ gives $\h_\Sigma(\mu) \leq \f(\mu)$. If $\f(\mu) = -\infty$ then we immediately get $\h_\Sigma(\mu) = \f(\mu)$.
		
		\item For each $m$, for all $n > K_m$ we have $\sigma_n \not\in \calE^{\leq}_{m,n}$. So for all large $n$ we have $\sigma_n \not\in \calF^{\leq}(\varepsilon_m, \delta_m, n)$. If $\f(\mu) \geq 0$ this means
			\[ \liminf_{n \to \infty} \frac{1}{n} \log \abs*{\Omega(\sigma_n, \mu, \varepsilon_m)} \geq \f(\mu) - \delta_m , \]
		and taking $m \to \infty$ gives $\h_\Sigma(\mu) \geq \f(\mu)$.
		 
		 \item For each $m$, for all $n > K_m$ we have $\sigma_n \not\in \calE^3_{m,n}$. So $\Sigma$ is a sofic approximation to $\Gamma$. \qedhere
	\end{enumerate}

\section{Nonequilibrium over every sofic approximation}
\label{sec:thmBproof}
In this section, energy and pressure functions will refer to an Ising model with interaction strength $J>0$ and no external field.

In this section we establish that, for low enough temperatures, the free-boundary state on the rank-$r$ free group $\FF_r$ is nonequilibrium over every sofic approximation. We show that the temperature being at or below the reconstruction threshold is a sufficient condition for $r \geq 2$.

We use a $\Sigma$-independent upper bound for $\h_\Sigma$ to compare the free boundary state to a deterministic state, which has entropy zero over every $\Sigma$. The upper bound is the following:

\begin{lemma}
	Let $\Sigma$ be any sofic approximation to the free group $\FF_r = \langle s_1, \ldots, s_r \rangle$, and let $\mu \in \Prob^{\FF_r}(\A^{\FF_r})$. Let $\mu_k \in \Prob^{\ZZ}(\A^{\ZZ})$ be the marginal of $\mu$ on the copy of $\ZZ$ in $\FF_r$ generated by $s_k$. Then
		\[ \h_\Sigma(\mu) \leq \h^{\KS}(\mu_k) . \]
\end{lemma}

Here $\h^{\KS}$ refers to the standard Kolmogorov-Sinai entropy rate of a $\ZZ$-system.

This is very similar to \cite[Theorem 4.4]{bowen2020}, and has essentially the same proof. The difference is that result considers the action of a subgroup on the same space, while we restrict the space as well. Another way to look at it is that both sides of the inequality are sofic entropies of actions on the same space with the same partition (in this case, the partition generated by the labels at the identity), which is dynamically generating under the action of the whole group but not the subgroup: to prove \cite[Theorem 4.4]{bowen2020} one would need to work with a partition (or pseudometric) which is dynamically generating under the action of the subgroup.

It also follows from an upper bound for Rokhlin entropy \cite[Theorem 1.5]{seward2017}; Seward notes that the sofic version had been noticed independently by several authors, but I am not aware of an existing published reference.

\begin{proof}
	Let $\Sigma_k$ be the restriction of $\Sigma$ to $\langle s_k \rangle \cong \ZZ$. This is a sofic approximation, and any good model for $\mu$ over a map in $\Sigma$ is a good model for $\mu_k$ over the corresponding restricted map in $\Sigma_k$. Finally, by \cite{bowen2012}, the sofic entropy over $\Sigma_k$ is just the KS entropy rate.
\end{proof}

Let $\shent^{\edge}(\mu_k) = \shent_{\mb{X} \sim \mu_k}(X_1 | X_0)$ denote the ``conditional edge entropy.'' If $\mu_k$ is an invariant Markov chain then $\shent^{\edge}(\mu_k) = \h^{\KS}(\mu_k)$ \cite[Equation (4.25)]{cover2006}, and for general invariant measures we have $\shent^{\edge}(\mu_k) \geq \h^{\KS}(\mu_k)$. So, if we define
	\[ \shent_{\min}^{\edge}(\mu) = \min_k \ \shent^{\edge} (\mu_k) \]
then we get the following:

\begin{cor}
	Let $\Sigma$ be any sofic approximation to the free group $\FF_r = \langle s_1, \ldots, s_r \rangle$, and let $\mu \in \Prob^{\FF_r}(\A^{\FF_r})$. Then
		\[ \h_\Sigma(\mu) \leq \shent_{\min}^{\edge}(\mu). \]
\end{cor}

Define $\press^{\edge}(\cdot) = \shent_{\min}^{\edge}(\cdot) - u(\cdot)$. By the above, this is an upper bound for $\press_\Sigma(\cdot)$ for any $\Sigma$.

\begin{prop}
	Let $\delta_+$ be the point mass at the all-$(+1)$ configuration on $\{\pm 1\}^{\FF_r}$ and let $\Sigma$ be any sofic approximation to $\FF_r$. Then
		\[ \press_\Sigma(\delta_+) > \press^{\edge}(\mu^{\FB}) \]
	if and only if
		\[ r > \frac{\phi(1+e^{2J}) - \phi(e^{2J})}{2J} \]
	where
		\[ \phi(t) = t \log t . \]
\end{prop}

\begin{proof}
	Note
		\[ \h_\Sigma (\delta_+) = 0 \quad \text{and} \quad u(\delta_+) = -J r\]
	and, as we calculated above when considering $\f(\mu_t)$,
		\[ \shent^{\edge} (\mu^{\FB}) = \shent \left( \tfrac{1}{1+e^{2J}} \right) \quad \text{and} \quad u(\mu^{\FB}) = Jr \tfrac{ 1 - e^{2J}}{1+e^{2J}} .\]
	So the desired inequality $\press_\Sigma(\delta_+) > \press^{\edge}(\mu^{\FB})$ can be written out explicitly as
	\begin{align*}
		0- (-Jr) &> \left( - \left(\tfrac{1}{1+e^{2J}} \right) \log \left( \tfrac{1}{1+e^{2J}} \right) -\left(\tfrac{e^{2J}}{1+e^{2J}} \right) \log \left( \tfrac{e^{2J}}{1+e^{2J}} \right)\right) - Jr \tfrac{ 1 - e^{2J}}{1+e^{2J}} \\
	\intertext{which we can rearrange to}
		Jr (1+e^{2J}) + Jr (1- e^{2J}) &>  \log( 1+e^{2J}) - e^{2J} \log e^{2J} + e^{2J} \log (1+e^{2J}) \\
	\intertext{and then to the claimed inequality}
		2Jr &> \phi(1+e^{2J}) - \phi(e^{2J}). \qedhere
	\end{align*}
\end{proof}

Since $\phi$ is convex and $\log$ is concave,
	\[  \frac{\phi(1+e^{2J}) - \phi(e^{2J})}{2J} \leq \frac{1 \cdot \phi'(1+e^{2J})}{2J} = \frac{\log(1+e^{2J}) + 1}{2J} \leq \frac{\log(e^{2J}) + 1 \cdot \frac{1}{e^{2J}} + 1}{2J} \]
and a sufficient condition for $\mu^{\FB}$ to be nonequilibrium over every $\Sigma$ is therefore
 	\[ r > 1 + \frac{1+e^{-2J}}{2J} . \tag{$\ast$}\]
This is easier to work with while still being strong enough to prove the following restatement of Theorem \ref{thm:noneq}:

\begin{theorem}
	If $r \geq 2$ and either \[ J \geq 2 \cdot J_{uniq}(r) = 2 \arctanh \frac{1}{2r-1} = \log\frac{r}{r-1} \] or  \[ J \geq J_{rec}(r) = \arctanh [ (2r-1)^{-1/2}] \] then $\mu^{\FB}$ is nonequilibrium over every sofic approximation to $\FF_r$.
\end{theorem}

\begin{proof}
	Define
		\[ \rho(J) = 1 + \frac{1+e^{-2J}}{2J} . \]
	First consider the condition involving $J_{uniq}(t)$. 
	It is easy to see that $\rho(J)$ is decreasing, so it suffices to show that
		\[ r > \rho(2 J_{uniq}(r)) , \quad r \geq 2 . \]
	To check that inequality, we rearrange it as follows:
	\begin{align*}
		r &> 1 + \frac{1 + (\frac{r-1}{r})^2}{2 \log \frac{r}{r-1}} \\
		&= 1 + \frac{1 + \frac{r^2 - 2r + 1}{r^2}}{2 \log \frac{r}{r-1}} \\
		r^3 &> r^2 + \frac{2r^2 - 2r + 1}{2 \log \frac{r}{r-1}} \\
		r^2 (r-1) 2 \log \frac{r}{r-1} &> 2r(r-1)+1 \\
		2r(r-1) \left( r \log \frac{r}{r-1} - 1\right) &> 1.
	\end{align*}
	Now using a standard Taylor series, for $r \geq 2$
		\[ r \log \frac{r}{r-1} = -r \log(1 - \tfrac{1}{r}) = r \left( \tfrac{1}{r} + \tfrac{1}{2r^2} + \cdots \right) > 1 + \tfrac{1}{2r} \]
	so
		\[ 2r(r-1) \left( r \log \frac{r}{r-1} - 1\right) > r-1 \]
	and in particular the desired inequality holds for $r \geq 2$.
	
	Now consider the reconstruction threshold condition. Since $\rho(J)$ is a decreasing bijection $(0,\infty) \to (1,\infty)$, it has a decreasing inverse function $\rho^{-1} \colon (1,\infty) \to (0,\infty)$. Our sufficient condition ($\ast$) can then be written as $J > \rho^{-1}(r)$. We want to show that $J_{rec}(r) > \rho^{-1}(r)$ for $r \geq 2$. It suffices to show that
		\[ J_{rec}^{-1}(J) > \rho(J) \quad \text{for} \quad 0 < J \leq J_{rec}(2) = \arctanh (\tfrac{1}{\sqrt{3}}) : \tag{$\dagger$}\]
	if we can show this, then given $r \geq 2$ we have $0 < J_{rec}(r) \leq J_{rec}(2)$ so
	\begin{align*}
		J_{rec}^{-1}( J_{rec}(r) ) &> \rho( J_{rec}(r) ) \\
		r &> \rho( J_{rec}(r) )\\
		\rho^{-1}(r) &< J_{rec}(r),
	\end{align*}
	where in the last line we use that $\rho^{-1}$ is decreasing.
	
	Now, to show ($\dagger$), first rearrange as follows:
	\begin{align*}
		J_{rec}^{-1}(J) &> \rho(J) \\
		\frac{1}{2 \tanh^2 J} + \frac{1}{2} &> 1 + \frac{1+e^{-2J}}{2J} \\
		\frac{1}{\tanh^2 J} &> \frac{J + 1 + e^{-2J}}{J} \\
		J \cosh^2 J &> \sinh^2 J \, (J+1 + e^{-2J}) \\
		J (e^{2J} + e^{-2J} + 2) &> (e^{2J} + e^{-2J} - 2) (J+1 + e^{-2J}) \\
		4J + 2 &> e^{2J} + e^{-2J} + 1 + e^{-4J} - 2e^{-2J}\\
		4J+2 &>  2 \cosh 2J + (e^{-2J} - 1)^2 .
	\end{align*}
	Now it is easy to check that the right-hand side is a strictly convex function of $J$, that both sides are equal to 2 when $J=0$, and that strict inequality holds when $J = J_{rec}(2)$. It  follows that the inequality holds on $(0, J_{rec}(2)]$.
\end{proof}

\textbf{Remarks.}
\begin{enumerate}
	\item It is easy to see numerically that $J_{rec}(r) > 2 J_{uniq}(r)$ for $r \geq 3$. In particular, an alternative proof of the $J_{rec}$ condition could be to verify that inequality and then check the $r=2$ case separately. Even though the $J_{rec}$ condition gives a weaker result for most values of $r$, we include it to better complement Theorem~\ref{thm:existence}.
	\item The constant 2 in the condition $J \geq 2 J_{uniq}(r)$ does not appear to be optimal. Numerically, it seems that without the reduction to the simpler function $\rho(J)$ it is possible to reduce the condition to $J \geq 1.65 J_{uniq}(r)$. If we only ask for a condition for sufficiently large $r$ (instead of all $r \geq 2$) then it seems possible to reduce the constant down to at least 1.39. Ideally that constant would be 1, but it does \emph{not} appear that this method (using this particular upper bound on sofic entropy to compare to $\delta_+$) could achieve that.

Comparison with the plus state $\mu^+$ (which is equilibrium by Theorem \ref{thm:isingeq}) instead of $\delta_+$ appears more promising in the high-degree limit; see Figure~\ref{fig:pressurecomp}. More specifically, the following appears to be true:

\begin{conj}
	For every $\varepsilon > 0$, for all large enough $r$ if $J \geq (1+\varepsilon) J_{uniq}(r)$ then $\press^{\edge}(\mu^{\FB}) < \press_\Sigma(\mu^+)$. In particular, $\mu^{\FB}$ is $\Sigma$-nonequilibrium for every $\Sigma$.
\end{conj}

It may be possible that something stronger is true: that $\mu^{\FB}$ is always nonequilibrium for $J > J_{uniq}$ as long as $r \geq 2$. But the upper bound $\press^{\edge}$ does not seem sharp enough to achieve this.
\end{enumerate}

\begin{figure}
\includegraphics{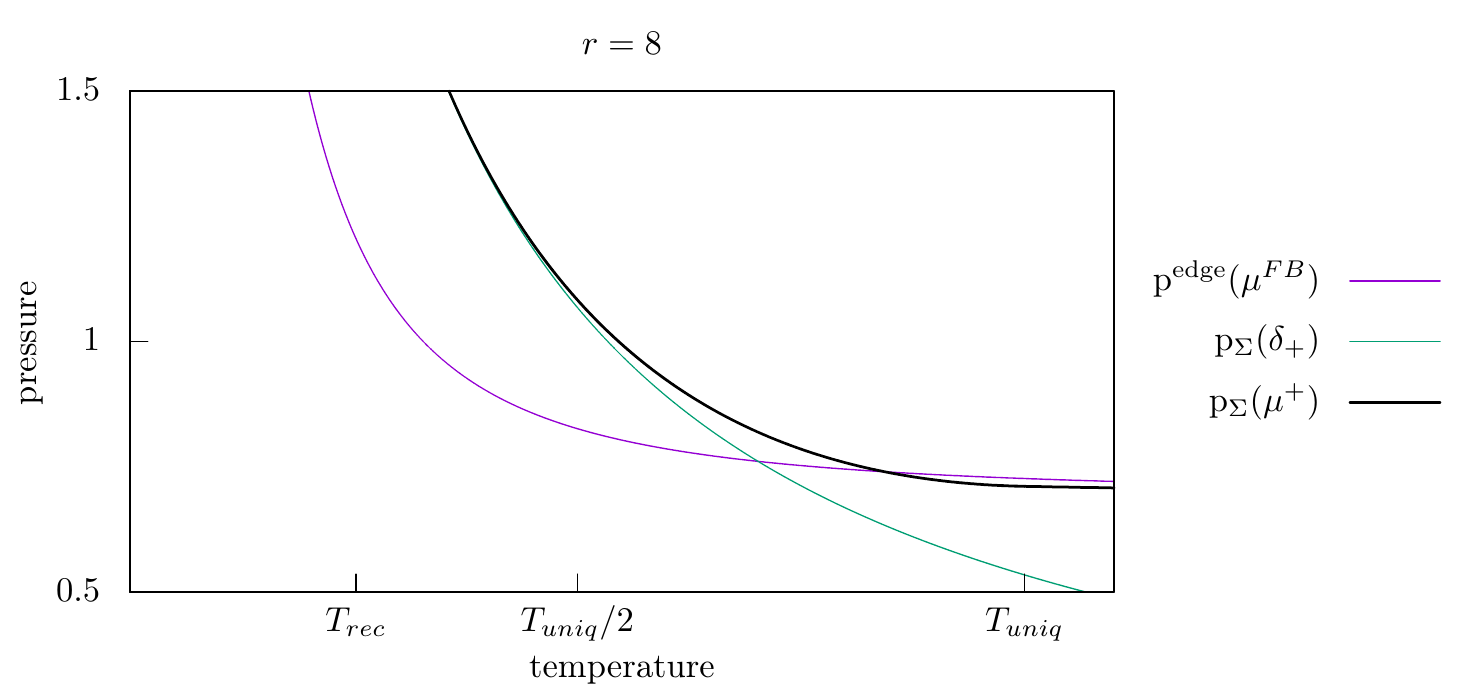}
\includegraphics{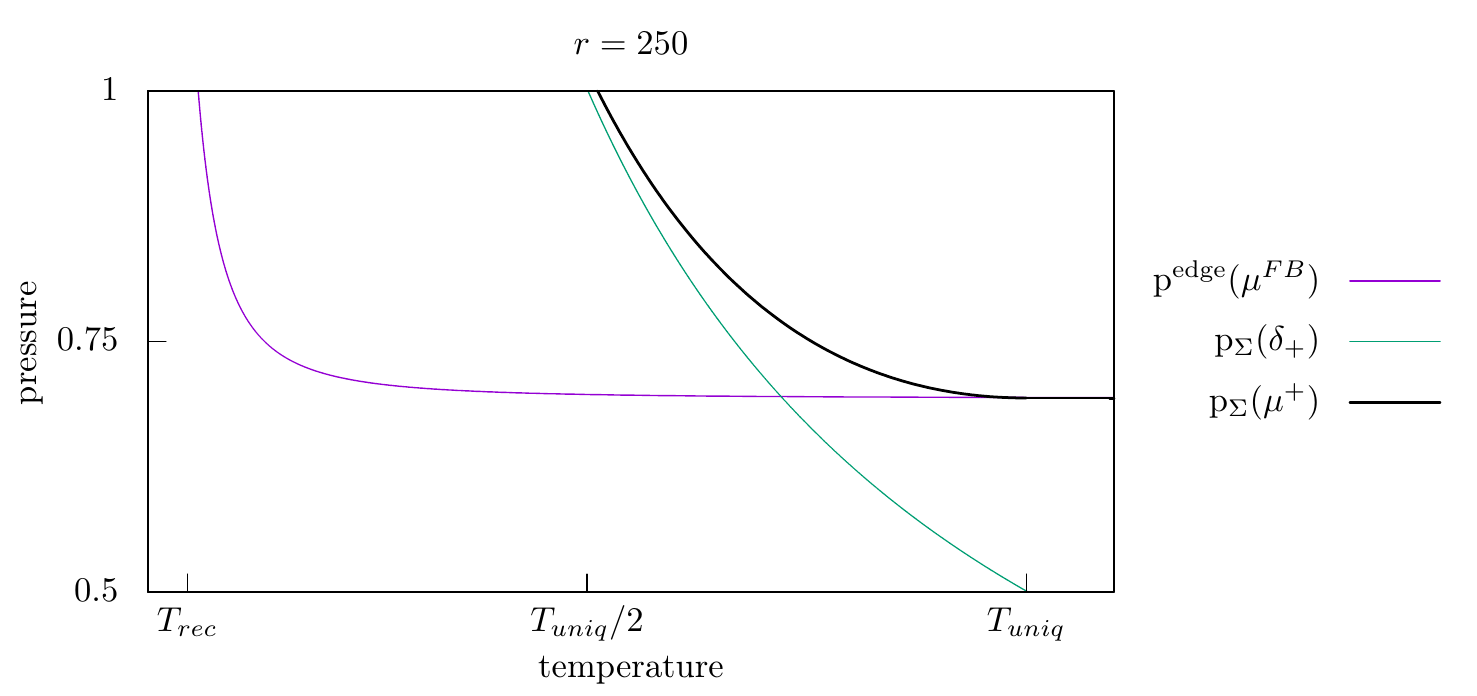}
\caption{Comparison of upper bound $\press^{\edge}$ for pressure of free boundary state with pressures of $\delta_+$ and $\mu^+$ (which is an equilibrium state, so has maximal pressure). Note that the latter two quantities are $\Sigma$-independent. In the notation above, the temperature is $T = 1/J$.
When the rank $r$ is higher, it seems like $\press^{\edge}(\mu^{\FB}) < \press_\Sigma(\mu^+)$ for temperatures much closer to the uniqueness threshold.}
\label{fig:pressurecomp}
\end{figure}

\section{Local limits and equilibrium measures (Proof of Theorem \ref{thm:isingeq})}

Recall that if $\sigma \colon \Gamma \to \Sym(V)$ is a map and $\calO \subset \Prob(\A^\Gamma)$ is any set then we defined $\Omega(\sigma, \calO) = \{ \mb{x} \in \A^V \st P_{\mb{x}}^\sigma \in \calO\}$ to be the set of ``microstates'' (labelings of $V$ by symbols in $\A$) such that the empirical distribution is in $\calO$ (which is typically an open neighborhood of some measure $\mu$ of interest).

Similarly, if $\zeta \in \Prob(\A^V)$ is the law of a random microstate then we can define the average empirical distribution
	\[ P_\zeta^\sigma = \sum_{\mb{x} \in \A^V} P_{\mb{x}}^\sigma \cdot \zeta\{\mb{x}\} \]
and define
	\[ \bOmega(\sigma, \calO) = \{ \zeta \in \Prob(\A^V) \st P_\zeta^\sigma \in \calO \} \]
to be the set of laws of random microstates whose average empirical distribution is in $\calO$. Note that we do \emph{not} require $\zeta$ to be supported on microstates whose individual empirical distributions are in $\calO$: in some cases below there may be no such microstates.

\begin{defn}
	Suppose $\Sigma = \{ \sigma_n \colon \Gamma \to \Sym(V_n)\}$ is a random sofic approximation to $\Gamma$. For each $n \in\NN$ and open neighborhood $\calU \supset \calX$, let $\xi_n^\calU = \widetilde{\xi}_{\sigma_n}^\calU \in \Prob(\A^{V_n})$ denote the restricted finitary Gibbs states for a potential $u$ (with some continuous extension $\tilde{u} \colon \A^\Gamma \to \RR$ if necessary).
	
	We say that $\mu$ is the local weak limit over $\Sigma$ of the finitary Gibbs states for $u$, or, more briefly $\lwlim(u,\Sigma)= \mu$, if for every open neighborhood $\calO \ni \mu$ 
		\[ \limsup_{\calU \downarrow \Prob^\Gamma(\calX)} \limsup_{n \to \infty} \frac{1}{\abs{V_n}} \log \PP\{ \xi_n^\calU \not\in \bOmega(\sigma_n,\calO) \} = -\infty . \]
	The $\limsup$ over $\calU$ here is over the directed set of open neighborhoods of $\Prob^\Gamma(\calX)$.
\end{defn}

Note that some call this notion ``local convergence on average.'' We require a superexponentially small probability of the empirical measure not being close to $\mu$ since even exponentially small events can affect the average pressure; see the proof of the third part of Lemma~\ref{lem:hmodproperties}.

\begin{lemma}
\label{lem:uniquelimits}
	For any fixed $u$, if $\lwlim(u,\Sigma)$ exists for every deterministic sofic approximation $\Sigma$, then there is some $\bar\mu$ such that $\lwlim(u,\Sigma) = \bar\mu$ for every sofic approximation $\Sigma$, random or deterministic.
\end{lemma}
\begin{proof}
	Suppose there are distinct $\bar\mu_1 \ne \bar\mu_2$ and deterministic sofic approximations $\Sigma_1 = (\sigma^1_n)_{n=1}^\infty, \Sigma_2 = (\sigma^2_n)_{n=1}^\infty$ to $\Gamma$ such that $\lwlim(u,\Sigma_i) = \mu_i$ for $i=1,2$. Then if $\tilde\Sigma = (\sigma^1_1, \sigma^2_1, \sigma^1_2, \sigma^2_2, \ldots)$ then $\lwlim(u, \tilde\Sigma)$ does not exist. This contradicts the premise that the limit exists over every deterministic $\Sigma$, so there must be some $\bar\mu$ such that $\lwlim(u, \Sigma) = \bar\mu$ if $\Sigma$ is any deterministic sofic approximation to $\Gamma$.
	
	In the deterministic case, ``$\lwlim(u,\Sigma)=\mu$'' means that for any $\calO \ni \mu$, for all small enough $\calU \supset \Prob^\Gamma(\calX)$ we have $\xi_n^\calU \in \bOmega(\sigma_n, \calO)$ for all large enough $n$.
	
	So for every open $\calO \ni \bar\mu$, for all small enough $\calU \supset \Prob^\Gamma(\calX)$ there exist $F,\delta$ such that if $\sigma$ is $(F,\delta)$-sofic then $\xi_\sigma^\calU \in \bOmega(\sigma_n, \calO)$. But if $\Sigma$ is a random sofic approximation then, by definition, $\PP\{ \sigma_n \text{ is $(F,\delta)$-sofic} \}$ approaches 1 superexponentially fast. The result follows.
\end{proof}

Here we use the formulation of the ergodic decomposition of $\mu \in \Prob^\Gamma(\A^\Gamma)$ as a measure with barycenter $\mu$, found for example in \cite[Section 12]{phelps2001}: the ergodic decomposition $\theta$ is the unique element of $\Prob(\Prob^\Gamma(\A^\Gamma))$ such that
	\[ \mu(f) = \int \nu(f)\, \theta(d \nu) \quad \text{ for all continuous } f \colon \A^\Gamma \to \RR \]
and $\theta(E) = 0$ for any measurable set containing no ergodic measures.\footnote{The reference \cite{phelps2001} refers to the Baire rather than Borel $\sigma$-algebra, but these coincide on $\A^\Gamma$ \cite[Exercise 7.2.8]{cohn2013}.}

\begin{defn}
	A (random or deterministic) sofic approximation $\Sigma$ to $\Gamma$ is ``\good{}'' if for any $\mu \in \Prob^\Gamma(\A^\Gamma)$, if $\theta$ is the ergodic decomposition of $\mu$ then
		\[ \h_\Sigma(\mu) \leq \int \h_\Sigma(\nu)\, \theta(d\nu) . \]
\end{defn}

For example, $\f$ is \good{} because the $\f$-invariant satisfies this inequality \cite{seward2016}. Also, if $\Sigma$ is an expander sequence then $\Sigma$ is \good{}: in this case, the left-hand side in the definition is $-\infty$ unless $\mu$ is ergodic.

Since the Kolmogorov-Sinai entropy rate of a shift system over an amenable group satisfies the ergodic decomposition formula \cite[Chapter 5]{moulinollagnier1985}, every sofic approximation to an amenable group is good.

Sofic approximations without this property do exist. The following example was suggested by Lewis Bowen:

\begin{example}
	Let $\mu_1$ be an ergodic measure with some (deterministic) $\Sigma_1$ such that $\h_{\Sigma_1}(\mu_1) = -\infty$; for example we could take $\mu_1$ to be the free-boundary state for the Ising model on a free group $\Gamma$ at temperature low enough for the $\f$-invariant to be negative. Let $\Sigma_2$ be another sofic approximation with $\h_{\Sigma_2}(\mu_1) \geq 0$, and let $\mu_2$ be a Bernoulli shift on $\{\pm 1\}^\Gamma$ (which has positive entropy over any sofic approximation \cite[Proposition 2.2]{bowen2010b}). Let $\mu = \frac{1}{2}(\mu_1+\mu_2)$. Note that since $\mu_1,\mu_2$ are ergodic this is the ergodic decomposition of $\mu$. Let $\Sigma$ be the disjoint union of $\Sigma_1, \Sigma_2$: if we think of sofic approximations as convergent graph sequences, then taking unions at each step of the sequence produces a new sequence with the same limit.
	
	Assuming that the finite sets $V_n$ along the sofic approximations are of essentially the same size, by ergodicity of $\mu_1$ and the assumption that $\h_{\Sigma_1}(\mu_1) = -\infty$, we also have $\h_\Sigma(\mu_1) = -\infty$. But $\h_\Sigma(\mu) \geq \frac{1}{2}(\h_{\Sigma_2}(\mu_1) + \h_{\Sigma_1}(\mu_2)) > -\infty$, so $\Sigma$ is not \good{}. \hfill $\triangleleft$
\end{example}

The following theorem is the main result of this section. It is used in Section~\ref{sec:eqconsequence} below to prove Theorem \ref{thm:isingeq}.

\begin{theorem}
\label{thm:eqgeneral}
	Suppose that for some random sofic approximation $\Sigma$ we have $\lwlim(u,\Sigma) = \bar\mu$. Then
	\begin{enumerate}
		\item $\bar\mu$ is a mixture of equilibrium states: There is some $\mathfrak{p} \in \Prob(\Prob^\Gamma(\calX))$ with barycenter $\bar\mu$ such that $\mathfrak{p}$-a.e.~$\nu$ is $\Sigma$-equilibrium.
		\item The maximal value of $\press_\Sigma(\cdot)$, attained by equilibrium measures, is
		\[ \lim_{\calU \downarrow \Prob^\Gamma(\calX)} \limsup_{n \to \infty} \frac{1}{\abs{V_n}} \log \EE_{\sigma_n} \exp \{ \widetilde{\Press}(\xi_n^\calU) \} . \]
		\item If $\Sigma$ is good, then the ergodic decomposition of $\bar\mu$ gives full measure to the set of $\Sigma$-equilibrium measures.
	\end{enumerate}
\end{theorem}

\noindent\textbf{Remarks.}
\begin{itemize}
	\item If the local limit $\bar\mu$ happens to be ergodic then the theorem implies that it is $\Sigma$-equilibrium, without having to assume $\Sigma$ is \good.

\item If $\bar\mu$ is \emph{not} ergodic, then it is not necessarily equilibrium: It will not be $\f$-equilibrium because of the ``error term'' in Seward's ergodic decomposition formula \cite{seward2016}. This can also happen with deterministic sofic approximations: if $\Sigma$ is an expander sequence and $\bar\mu$ is not ergodic then $\h_\Sigma(\bar\mu) = -\infty$ and in particular $\bar\mu$ is not equilibrium. 

\item If $\Sigma$ is deterministic, then the formula for the maximal pressure value reduces to Chung's variational principle \cite{chung2013}.
\end{itemize}

In the proof of the theorem we will use the following ``modified'' entropy and pressure: define
	\[ \h_{\Sigma}^{\ann}(\mu) = \inf_{\calO \ni \mu} \limsup_{n \to \infty} \frac{1}{\abs{V_n}} \log \EE_{\sigma_n} \exp \sup\left\{ \shent(\zeta) \st \zeta \in \bOmega(\sigma_n, \calO) \right\}. \]
If $\mu$ is ergodic, then any $\zeta \in \bOmega(\sigma_n, \calO)$ is mostly supported on good models for $\mu$, so the supremum over Shannon entropies is essentially $\log \abs*{\Omega(\sigma_n, \calO)}$, and $\h_\Sigma^{\ann}(\mu) = \h_\Sigma(\mu)$. For non-ergodic $\mu$, this modified quantity works better for our purposes below than $\h_\Sigma$. We similarly define the pressure as
	\[ \press_{\Sigma}^{\ann}(\mu) =  \lim_{ \calO \downarrow \mu} \limsup_{n \to \infty} \frac{1}{\abs{V_n}} \log \EE_{\sigma_n} \exp \sup \left\{ \widetilde{\Press}(\zeta) \st \zeta \in \bOmega(\sigma_n, \calO) \right\}. \]
Since we assume the specific energy $u$ is continuous, $\press_{\Sigma}^{\ann}(\mu) = \h_{\Sigma}^{\ann}(\mu) - u(\mu)$.

Similar notions of entropy have been considered by Alpeev \cite{alpeev2016}, Austin \cite[Section 6]{austin2016}, and Bowen \cite{bowen2011}.

We need the following properties of $\h^{\ann}$:
\begin{lemma}
\label{lem:hmodproperties}\ 
	\begin{enumerate}
		\item For every $\mu \in \Prob^\Gamma(\A^\Gamma)$, there is some $\mathfrak{p} \in \Prob(\Prob^\Gamma(\A^\Gamma))$ with barycenter $\mu$ such that
				\[ \h_\Sigma^{\ann}(\mu) \leq \int \h_\Sigma(\nu) \, \mathfrak{p}(d\nu) . \]
			If $\Sigma$ is \good{} then we can take $\mathfrak{p}$ to be the ergodic decomposition of $\mu$.
		\item For any $\mu$, $\h_\Sigma^{\ann}(\mu) \geq \h_\Sigma(\mu)$.
		\item If $\lwlim(u,\Sigma) = \bar\mu$, then $\bar\mu$ maximizes $\press_{\Sigma}^{\ann}$ on $\Prob^\Gamma(\calX)$, and
			\[ \press_{\Sigma}^{\ann}(\bar\mu) = \lim_{\calU \downarrow \Prob^\Gamma(\calX)} \limsup_{n \to \infty} \frac{1}{\abs{V_n}} \log \EE_{\sigma_n} \exp\{ \widetilde{\Press}(\xi_n^\calU)\} . \]
	\end{enumerate}
\end{lemma}

First we show that the lemma implies the theorem. The proof of the lemma is in Section \ref{sec:lemmaproof}.

\begin{proof}[Proof of Theorem \ref{thm:eqgeneral}]
	Part 1 of the lemma implies that
		\[ \press_{\Sigma}^{\ann}(\bar\mu) \leq \int \press_{\Sigma}(\nu) \, \mathfrak{p}(d\nu) , \]
	where we can take $\mathfrak{p}$ to be the ergodic decomposition of $\bar\mu$ if $\Sigma$ is good. Note that $\mathfrak{p} \in \Prob(\Prob^\Gamma(\A^\Gamma))$ must give full measure to $\Prob^\Gamma(\calX)$ in order to have barycenter $\mu \in \Prob^\Gamma(\calX)$: let $\phi \colon \A^\Gamma \to \RR$ denote the distance to $\calX$. Since $\calX$ is closed, $\{ \phi = 0\} = \calX$. Then if $\mathfrak{p}$ has barycenter $\mu$
		\[ \int \nu(\phi) \, \mathfrak{p}(d\nu) = \mu(\phi) = 0 \]
	so for $\mathfrak{p}$-a.e. $\nu$ we have $\nu(\phi) = 0$, which means $\nu$-a.e. $\mathbf{x}$ is in $\calX$.
	
	So since, by parts 2 and 3 of the lemma, for any $\nu \in \Prob^\Gamma(\calX)$ we have $\press_\Sigma(\nu) \leq \press_\Sigma^{\ann}(\nu) \leq \press_\Sigma^{\ann}(\bar\mu)$, in fact $\press_{\Sigma}(\nu) = \press^{\ann}_{\Sigma}(\bar\mu)$ for $\mathfrak{p}$-a.e.~$\nu$.
	
	On the other hand, suppose $\nu$ is a measure which is not $\Sigma$-equilibrium. Then we can pick some $\mu$ with $\press_{\Sigma}(\mu) > \press_{\Sigma}(\nu)$, and using parts 1 and 2 of the Lemma we get the chain of inequalities
		\[ \press_{\Sigma}^{\ann}(\bar\mu) \geq \press_{\Sigma}^{\ann} (\mu) \overset{\text{part 2}}{\geq} \press_{\Sigma}(\mu) > \press_{\Sigma}(\nu), \]
	so in particular $\press_{\Sigma}^{\ann}(\bar\mu) \ne \press_{\Sigma}(\nu)$. But this is not the case for $\mathfrak{p}$-a.e. $\nu$, so $\mathfrak{p}$-a.e. $\nu$ is $\Sigma$-equilibrium. This proves parts 1 and 3 of the theorem.
	
	This fact, along with the above conclusion that $\press_{\Sigma}(\nu) = \press^{\ann}_{\Sigma}(\bar\mu)$ for $\mathfrak{p}$-a.e.~$\nu$, implies that $\press^{\ann}_{\Sigma}(\bar\mu)$ is the value of $\press_{\Sigma}(\nu)$ for $\Sigma$-equilibrium $\nu$. So the formula for $\press^{\ann}_{\Sigma}(\bar\mu)$ in item 3 of the Lemma is a formula for the maximal value of $\press_\Sigma(\cdot)$: this proves part 2 of Theorem~\ref{thm:eqgeneral}.
\end{proof}

The lemma also implies the following:
\begin{prop}
	If $\Sigma$ is deterministic, then $\h_\Sigma^{\ann}(\cdot)$ is the upper concave envelope of $\h_\Sigma(\cdot)$.
\end{prop}
\begin{proof}
	Recall that the upper concave envelope of a function $f$ is the pointwise infimum of the set of continuous affine functions bounded below by $f$ \cite[Section 3]{phelps2001}. Equivalently, it is the unique upper semicontinuous, concave function bounded below by $f$ and bounded above by every continuous affine function bounded below by $f$.
	
	Suppose $h \colon \Prob^\Gamma(\calX) \to \RR$ is a continuous affine function with $h \geq \h_\Sigma$. Then by part 1 of the lemma
	\begin{align*}
		\h_\Sigma^{\ann}(\mu)
			&\leq \int \h_\Sigma(\nu) \, \mathfrak{p}(d\nu) \\
			&\leq \int h(\nu) \, \mathfrak{p}(d\nu) \\
			&= h(\mu).
	\end{align*}
	So $\h_\Sigma^{\ann}$ is bounded above by the upper envelope of $\h_\Sigma$. This is true even if $\Sigma$ is random.
	
	Note also that $\h_\Sigma^{\ann}$ is upper semicontinous and that by part 2 of the lemma $\h_\Sigma^{\ann} \geq \h_\Sigma$. To finish the proof we note that $\h_\Sigma^{\ann}$ is concave if $\Sigma$ is deterministic, which follows from concavity of Shannon entropy.
\end{proof}


\subsection{Proof of Theorem \ref{thm:isingeq}}
\label{sec:eqconsequence}
Whether or not $\Sigma$ is good, in the case of the Ising model on a free group we can then take advantage of symmetry to prove Theorem \ref{thm:isingeq}:
\begin{theorem}
	Let $\Gamma$ be a free group and $\Sigma$ be any sofic approximation to $\Gamma$. For an Ising model on $\Gamma$ at any positive temperature and zero external field, both $\mu^+$ and $\mu^-$ are $\Sigma$-equilibrium.
\end{theorem}
\begin{proof}
	Since $\mu^+$, $\mu^-$ are distinct ergodic measures, they are mutually singular. Let $E^+$ be a shift-invariant set with $\mu^+(E^+) = 1$ and $\mu^-(E^+)=0$. Define $\Phi \colon \{\pm 1\}^\Gamma \to \{\pm 1\}^\Gamma$ by
		\[ \Phi(\mb{x}) = \left\{ \begin{array}{ll}
				\mb{x}, & \mb{x} \in E^+ \\
				-\mb{x}, & \mb{x} \not\in E^+
				\end{array} \right. \]
	Then $\Phi_*\mu^+ = \mu^+$ since $\Phi(\mb{x}) = \mb{x}$ $\mu^+$-a.s., and $\Phi_*\mu^- = \mu^+$ since $\Phi(\mb{x}) = -\mb{x}$ $\mu^-$-a.s. and each of $\mu^+$ and $\mu^-$ is the image of the other under flipping signs. Therefore if $\mu$ is any convex combination of $\mu^+,\mu^-$ we have $\Phi_*\mu = \mu^+$. Since $\Phi$ is finite-to-1 it does not decrease sofic entropy \cite[Proposition 5.7]{bowen2020}, and by symmetry it does not change the energy. Therefore $\Phi$ does not decrease pressure, and it suffices to find some $\Sigma$-equilibrium $\mu$ which is a convex combination of $\mu^+, \mu^-$.
	
	By \cite{montanari2012}, $\lwlim(u,\Sigma) = \bar\mu = \frac{1}{2}(\mu^+ + \mu^-)$. Since there is some $\mathfrak{p}$ with barycenter $\bar\mu$ such that $\mathfrak{p}$-a.e.~$\nu$ is equilibrium (by Theorem \ref{thm:eqgeneral}), this finishes the proof (if $\mathfrak{p}$ has barycenter $\bar\mu$, $\mathfrak{p}$-a.e.~$\nu$ must be a convex combination of $\mu^+, \mu^-$, the ergodic components of $\bar\mu$).
\end{proof}

It was also claimed in Theorem~\ref{thm:isingeq} that, with zero external field, $\mu^+$ and $\mu^-$ are the only $\f$-equilibrium states. This follows from the following lemma, Proposition~\ref{prop:fnoneq}, and the fact that $\mu^+$, $\mu^-$, and $\mu^{\FB}$ are the only Ising Gibbs states which are completely homogeneous Markov chains \cite[Section 12.2]{georgii2011}. A tree-indexed Markov chain is called completely homogeneous if the transition kernel is the same across every edge.

The lemma is stated in slightly further generality than is required for the Ising model. This requires making one more definition: we say that a subshift $\calX \subset \A^{\FF_r}$ is a \emph{topological Markov chain} if there is some $\{0,1\}$-valued $\abs{\A} \times \abs{\A}$ matrix $M$ such that $\calX$ is the set of all $\mb{x} \in \A^{\FF_r}$ such that if $\gamma \in \FF_r$ and $i \in [r]$ then $M_{\mb{x}(\gamma),\mb{x}(s_i \gamma)} = 1$. In other words, there is some set of pairs of symbols in $\A$ (encoded by $M$) which are allowed to appear on either end of a directed edge, and $\calX$ is the set of \emph{all} microstates for which only allowed pairs appear.

Some authors refer to this as a ``subshift of finite type.''

\begin{lemma}
	Let $\calX$ be a subshift of some $\A^{\FF_r}$ which is a topological Markov chain. Suppose that $\mu \in \Prob^{\FF_r}(\calX)$ is $\f$-equilibrium for some ``nearest-neighbor'' energy function $u \colon \calX \to \RR$, i.e. one of the form
		\[ u(\mb{x}) = B(\mb{x}(e)) + \sum_{i = 1}^r J(\mb{x}(e), \mb{x}(s_i)) \]
	for functions $B \colon \A \to \RR$ and $J \colon \A^2 \to \RR$.
	
	Then $\mu$ is a completely homogeneous Markov chain.
\end{lemma}
\begin{proof}
	Let $\mu^1$ be the Markov approximation to $\mu$. Note that since $\calX$ is a topological Markov chain, $\mu^1$ is still supported on $\calX$. By \cite[Theorem 11.1]{bowen2010c}, $\f(\mu^1) \geq \f(\mu)$, with equality if and only if $\mu = \mu^1$, i.e. $\mu$ is a Markov chain. Since $u(\mu)$ only depends on the single-site and single-edge marginals, $u(\mu) = u(\mu^1)$. So $\press_{\f}(\mu^1) \geq \press_{\f}(\mu)$, with equality if and only if $\mu$ is a Markov chain. In particular, if $\mu$ is $\f$-equilibrium then it is a Markov chain.
	
	We have assumed that $\mu$ is shift-invariant, but we need to explain why it is completely homogeneous: a priori there may be a different Markov transition kernel for each generator $s_i$ (i.e. each edge type). To see that this is not possible for an equilibrium measure, note that (since $\mu$ is Markov) if we let $X_e, X_1, \ldots X_r$ denote random spins at the sites $e, s_1, \ldots s_r \in \FF_r$ distributed according to $\mu$ then
	\begin{align*}
		\press_{\f}(\mu) &= \f(\mu) - u(\mu)  \\
			&= \left[ (1-2r) \shent(X_e) + \sum_{i=1}^r \shent(X_e, X_i) \right] - \EE\left[ B(X_e) + \sum_{i=1}^r J(X_e, X_r) \right] \\
			&= \left[ (1-2r) \shent(X_e) - \EE B(X_e) \right] + \sum_{i=1}^r \left[ \shent(X_e, X_i) - \EE J(X_e, X_i)\right] .
	\end{align*}
	So whichever $i\in [r]$ corresponds to the largest term in the sum, we should just use that transition kernel for every edge type to maximize the pressure.
\end{proof}

\begin{cor}
	If $\lwlim(u,\f) = \bar\mu$, then $\bar\mu$ is a mixture of shift-invariant Markov chains which are Gibbs with respect to $u$.
\end{cor}
\begin{proof}
	The limit $\bar\mu$ is a mixture of $\f$-equilibrium states by Theorem \ref{thm:eqgeneral}. These must be Markov chains by the previous lemma. They must be Gibbs since $\bar\mu$ is an invariant Gibbs measure, and hence cannot be a mixture of non-Gibbs invariant measures \cite[Theorem 14.15(c)]{georgii2011} (note this theorem is stated for actions of $\ZZ^d$ but the same proof works for any countable group).
\end{proof}

Note that the assumption of the previous result holds in particular if the local limit exists over every deterministic sofic approximation, by Lemma \ref{lem:uniquelimits}. This explains the observation made in \cite{montanari2012} that it is possible to identify the limit of Ising models on locally tree-like graphs, not just prove that it exists: as long as the limit does exist (and is independent of the sofic approximation), we have shown here that it must be something simple.


\subsection{Proof of Lemma \ref{lem:hmodproperties}}
\label{sec:lemmaproof}
\subsubsection{Part 1}
Let
		\[  \h_\Sigma^{\ann}(\mu; k) = \limsup_{n \to \infty} \frac{1}{\abs{V_n}} \log \EE_{\sigma_n} \exp \sup\left\{ \shent(\zeta) \st \zeta \in \bOmega(\sigma_n, \ball{\mu}{1/k}) \right\} .\]
	Then, since the radius-$1/k$ balls form a basis for the topology at $\mu$,
	\begin{align*}
		\h_\Sigma^{\ann}(\mu)
			&= \inf_{\calO \ni \mu} \limsup_{n \to \infty} \frac{1}{\abs{V_n}} \log \EE_{\sigma_n} \exp \sup\left\{ \shent(\zeta) \st \zeta \in \bOmega(\sigma_n, \calO) \right\} \\
			&= \lim_{k \to \infty} \h_\Sigma^{\ann}(\mu; k).
	\end{align*}
	For each $n,k$, let $\zeta_{n,k} \in \bOmega(\sigma_n, \overline{\ball{\mu}{1/k}}) \subset \Prob(\A^n)$ such that
		\[ \shent(\zeta_{n,k}) = \sup\{\shent(\zeta) \st \zeta \in \bOmega(\sigma_n, \ball{\mu}{1/k})\} . \]
 Keep in mind that this depends on $\sigma_n$. Then
 			\[ \h_\Sigma^{\ann}(\mu; k) = \limsup_{n \to \infty} \frac{1}{\abs{V_n}} \log \EE_{\sigma_n} \exp \shent(\zeta_{n,k}) . \]
	
	Now let $I$ be a finite set and let $\pi \colon \Prob^\Gamma(\A^\Gamma) \to I$ be a measurable map. We think of this as creating a labeled finite partition of $\Prob^\Gamma(\A^\Gamma)$. For each $i \in I$ let $\zeta_{n,k}^i \in \Prob(\A^n)$ be $\zeta_{n,k}$ conditioned on the set $\{ \mb{x} \in \A^n \st P_{\mb{x}}^{\sigma_n} \in \pi^{-1}\{i\}\}$ (maintaining the implicit dependence on $\sigma_n$). Let
		\[ p_{n,k} = (\mb{x} \mapsto \pi (P^{\sigma_n}_{\mb{x}}))_* \zeta_{n,k} \in \Prob(I) \]
	or, in other words,
		\[ p_{n,k}(i) = \zeta_{n,k} \big\{ \mb{x} \in \A^n \st P_{\mb{x}}^{\sigma_n} \in \pi^{-1}\{i\} \big\} . \]
	Then $\zeta_{n,k} = \sum_{i \in I} p_{n,k}(i) \zeta_{n,k}^i$, and by the chain rule for conditional entropy
		\[ \shent(\zeta_{n,k}) = \sum_{i \in I} p_{n,k}(i) \shent(\zeta_{n,k}^i) + \shent(p_{n,k}) \leq \sum_{i \in I} p_{n,k}(i) \shent(\zeta_{n,k}^i) + \log\abs{I} . \]
	Combining with the above, it follows that
		\[ \h_\Sigma^{\ann}(\mu; k) \leq \limsup_{n \to \infty} \frac{1}{\abs{V_n}} \log \EE_{\sigma_n} \exp \sum_{i \in I} p_{n,k}(i) \shent(\zeta_{n,k}^i) . \]
	
	Let
		\[ \mathfrak{P}_k = \left\{ p \in \Prob(I) \st \exists \left\{ \mu_i \in \overline{\conv}(\pi^{-1}\{i\}) \right\}_{i \in I} \text{ s.t. } \sum_{i \in I} p(i) \mu_i \in \ball{\mu}{1/k} \right\} . \]
	(Here $\overline{\conv}$ is the closed convex hull. Later we will pick $\pi,I$ so that all of these have small diameter.) Then by definition $p_{n,k} \in \mathfrak{P}_k$ for every $n,k$: we can take $\mu_i = P_{\zeta_{n,k}^i}^{\sigma_n}$ since $P_{\zeta_{n,k}}^{\sigma_n} \in \ball{\mu}{1/k}$.
	
	Give $\Prob(I)$ the metric $d(p_1,p_2) = \max_{i \in I} \abs{p_1(i) - p_2(i)}$. Given $\delta>0$, since $\Prob(I)$ is compact we can cover $\mathfrak{P}_k$ with finitely many radius-$\delta$ balls $\{\ball{p_j}{\delta}\}_{j \in J}$ with $p_j \in \mathfrak{P}_k$ for all $j \in J$. Then
	\begin{align*}
		\EE_{\sigma_n} \exp \sum_{i \in I} p_{n,k}(i) \shent(\zeta_{n,k}^i)
			&\leq \EE_{\sigma_n} \sum_{j \in J} \1_{p_{n,k} \in \ball{p_j}{\delta}} \exp \sum_{i \in I} p_{n,k}(i) \shent(\zeta_{n,k}^i) \\
			&\leq \EE_{\sigma_n} \sum_{j \in J} \1_{p_{n,k} \in \ball{p_j}{\delta}}\cdot \exp \sum_{i \in I} (p_j(i)+\delta) \log\abs*{\Omega(\sigma_n, \pi^{-1}\{i\})} \\
			&\leq \EE_{\sigma_n} \sum_{j \in J} 1\cdot \exp \sum_{i \in I} (p_j(i)+\delta) \log\abs*{\Omega(\sigma_n, \pi^{-1}\{i\})} \\
			&= \sum_{j \in J} \EE_{\sigma_n} \exp \sum_{i \in I} (p_j(i)+\delta) \log\abs*{\Omega(\sigma_n, \pi^{-1}\{i\})} \\
			&\leq \abs{J} \max_{j \in J} \EE_{\sigma_n} \exp \sum_{i \in I} (p_j(i)+\delta) \log\abs*{\Omega(\sigma_n, \pi^{-1}\{i\})}
	\end{align*}
	so
	\begin{align*}
		\h_\Sigma^{\ann}(\mu;k)
			&\leq \limsup_{n \to \infty} \frac{1}{\abs{V_n}} \log \left(\abs{J} \max_{j \in J} \EE_{\sigma_n} \prod_{i \in I} \abs*{\Omega(\sigma_n, \pi^{-1}\{i\})}^{p_j(i)+\delta} \right) \\
			&= \max_{j \in J} \limsup_{n \to \infty} \frac{1}{\abs{V_n}} \log \EE_{\sigma_n} \prod_{i \in I} \abs*{\Omega(\sigma_n, \pi^{-1}\{i\})}^{p_j(i)+\delta} \\
			&= \sup_{p \in \mathfrak{P}_k} \limsup_{n \to \infty} \frac{1}{\abs{V_n}} \log \EE_{\sigma_n} \prod_{i \in I} \abs*{\Omega(\sigma_n, \pi^{-1}\{i\})}^{p(i)+\delta} \\
			&\leq \sup_{p \in \mathfrak{P}_k} \limsup_{n \to \infty} \frac{1}{\abs{V_n}} \log \prod_{i \in I} \left(\EE_{\sigma_n} \abs*{\Omega(\sigma_n, \pi^{-1}\{i\})}^{1+\delta/p(i)} \right)^{p(i)} \tag{H\"older} \\
			&\leq \sup_{p \in \mathfrak{P}_k} \limsup_{n \to \infty} \frac{1}{\abs{V_n}} \sum_{i \in I} p(i) \log \EE_{\sigma_n} \left(\abs*{\Omega(\sigma_n, \pi^{-1}\{i\})}\cdot \abs{\A^n}^{\delta/p(i)} \right) \\
			&\leq \sup_{p \in \mathfrak{P}_k} \sum_{i \in I} p(i) \limsup_{n \to \infty} \frac{1}{\abs{V_n}} \log \EE_{\sigma_n} \abs*{\Omega(\sigma_n, \pi^{-1}\{i\})} + \delta \abs{I} \log\abs{\A} .
	\end{align*}
	We can then take $\delta$ to 0. Writing
		\[ F(\pi^{-1}\{i\}) \coloneqq \limsup_{n \to \infty} \frac{1}{\abs{V_n}} \log \EE_{\sigma_n} \abs*{\Omega(\sigma_n, \pi^{-1}\{i\})} , \]
	we have that for any $I, \pi$
		\[ \h_\Sigma^{\ann}(\mu) \leq \lim_{k \to \infty} \sup_{p \in \mathfrak{P}_k} \sum_{i \in I} p(i) F(\pi^{-1}\{i\}) . \]
	Note that the supremum is decreasing in $k$, so the limit exists. For each $k$ we can choose $p_k \in \mathfrak{P}_k$  which gets within $\varepsilon$ of the supremum, then pick a convergent subsequence $p_{k_n} \to p \in \Prob(I)$ to get
		\[ \h_\Sigma^{\ann}(\mu) \leq \sup_{p \in \mathfrak{P}_I} \sum_{i \in I} p(i) F(\pi^{-1}\{i\}) \tag{*} \]
	where
		\[ \mathfrak{P}_I = \{ p \in \Prob(I) \st \exists \{\mu_i \in \overline{\conv}(\pi^{-1}\{i\}) \}_{i \in I} \text{ s.t. } \sum_{i \in I} p(i) \mu_i = \mu \}, \]
	and moreover there is some $p_I \in \mathfrak{P}_I$ which attains the supremum.
	
	Now take some sequence $I_1, I_2, \ldots$ with maps $\pi_k \colon \Prob^\Gamma(\A^\Gamma) \to I_k$ for each $k \in \NN$ such that $\max_{i \in I_k} \diam(\overline{\conv}(\pi_{k}^{-1}\{i\})) < 1/k$. For each $k$ and $i \in I_k$ let $\nu_{k,i} \in \overline{\pi_k^{-1}\{i\}}$ attain $\h_\Sigma(\nu_{k,i}) = F(\pi_k^{-1}\{i\})$. Define
		\[ \mathfrak{p}_k = \sum_{i \in I_k} p_{I_k}(i) \delta_{\nu_{k,i}} \in \Prob(\Prob^\Gamma(\A^\Gamma)) . \]
	By passing to a subsequence we can assume that $\{\mathfrak{p}_k\}$ has some weak limit $\mathfrak{p}$. By (*), definition of $\mathfrak{p}_k$, and using that $\h_\Sigma$ is upper semicontinuous,
		\[ \h_\Sigma^{\ann}(\mu) \leq \lim_{k \to \infty} \int \h_\Sigma \, d\mathfrak{p}_k
			\leq \int \h_\Sigma\, d\mathfrak{p}. \]
	
	We now show that $\mathfrak{p}$ has barycenter $\mu$; the result then follows, with the comment about the case when $\Sigma$ is \good{} following from Lemma~\ref{lem:sewgen}. First suppose that $g \colon \A^\Gamma \to \RR$ is 1-Lipschitz. For each $k \in \NN$ and $i \in I_k$ pick $\mu_{k,i} \in \overline{\conv}(\pi_k^{-1}\{i\})$ such that $\sum_{i \in I_k} p_{I_k}(i) \mu_{k,i} = \mu$ (using that $p_{I_k} \in \mathfrak{P}_{I_k}$). Then for each $k$ we have
		\[ \int \nu(g)\, \mathfrak{p}_k (d\nu) = \sum_{i \in I_k} p_{I_k}(i) \nu_{k,i}(g) \overset{1/k}{\approx} \sum_{i \in I_{k}} p_{I_k}(i) \mu_{k,i}( g) = \mu g . \]
	Taking $k$ to infinity, we get
		\[ \lim_{k \to \infty} \int \nu(g)\, \mathfrak{p}_k(d\nu) = \mu g . \]
	But since $\nu \mapsto \nu(g)$ is continuous and $\mathfrak{p}_k \to \mathfrak{p}$, this shows that $\int \nu (g)\, d\mathfrak{p} = \mu (g)$ for any 1-Lipschitz $g$. By monotone convergence, the same is true for any continuous $g$, and $\mu$ is the barycenter of $\mathfrak{p}$. \\

	If $\Sigma$ is \good{}, we can take $\mathfrak{p}$ to be the ergodic decomposition of $\mu$ by the following lemma:
\begin{lemma}
\label{lem:sewgen}
	Suppose $\mathfrak{p} \in \Prob(\Prob^\Gamma(\A^\Gamma))$ has barycenter $\mu$, and for each $\nu \in \Prob^\Gamma(\A^\Gamma)$ let $\mathfrak{e}_\nu \in \Prob(\Prob^\Gamma(\A^\Gamma))$ be its ergodic decomposition. Then for any \good{} $\Sigma$
		\[ \int \h_\Sigma(\nu)\, \mathfrak{p}(d\nu) \leq \int \h_\Sigma(\nu)\, \mathfrak{e}_\mu(d\nu) . \]
\end{lemma}

\begin{proof}
	Note $\int \h_\Sigma(q)\, \mathfrak{e}_\nu(dq)$ is a measurable function of $\nu$, for example using \cite[Proposition 11.1(ii)]{phelps2001} and that $\h_\Sigma$ is upper semicontinuous. Therefore, using that $\Sigma$ is \good{}, we can write
	\begin{align*}
		\int \h_\Sigma(\nu)\, \mathfrak{p}(d\nu)
			&\leq \int \left[ \int \h_\Sigma(q)\, \mathfrak{e}_\nu(dq) \right] \mathfrak{p}(d\nu).
	\end{align*}
	
	Now $g \mapsto  \int \left[ \int g(q)\, \mathfrak{e}_\nu(dq) \right] \mathfrak{p}(d\nu)$ is a positive linear functional on $C(\Prob^\Gamma(\A^\Gamma))$ so by Riesz representation there is some $\theta \in \Prob(\Prob^\Gamma(\A^\Gamma))$ such that
		\[  \int \left[ \int G(q)\, \mathfrak{e}_\nu(dq) \right] \mathfrak{p}(d\nu) = \int G(\nu)\, \theta(d\nu) \]
	for all $G \in C(\Prob^\Gamma(\A^\Gamma))$. The same holds for semicontinuous $g$ by monotone convergence; in particular it holds for $\h_\Sigma$.
	
	Now for any $g \in C(\A^\Gamma)$, since $\nu \mapsto \nu(g)$ is in $C(\Prob^\Gamma(\A^\Gamma))$ we have
	\begin{align*}
		\int \nu(g)\, \theta(d\nu)
			&= \int \left[ \int q(g)\, \mathfrak{e}_\nu(dq) \right] \mathfrak{p}(d\nu) \\
			&= \int \left[ \nu(g) \right] \mathfrak{p}(d\nu) \tag{barycenter of $\mathfrak{e}_\nu$ is $\nu$} \\
			&= \mu(g) \tag{barycenter of $\mathfrak{p}$ is $\mu$}
	\end{align*}
	Therefore the barycenter of $\theta$ is $\mu$.
	
	Let $\mathfrak{E}$ be the set of ergodic measures in $\Prob^\Gamma(\A^\Gamma)$. We now show that $\theta(\mathfrak{E}) = 1$, which by uniqueness implies $\theta = \mathfrak{e}_\mu$ (see for example the last theorem (unnumbered) in Section 12 of \cite{phelps2001} for the uniqueness theorem in this form). Since $\mathfrak{E}$ is a $G_\delta$ set, there exist open sets $G_1 \supset G_2 \supset \ldots$ such that $\mathfrak{E} = \bigcap_{n=1}^\infty G_n$. For each $n$ let $F_n \in C(\Prob^\Gamma(\A^\Gamma))$ be given by $F_n(\nu) = \bar{d}(\nu, G_n^c)$. Then $F_n \searrow \1_{\mathfrak{E}^c}$ so, using monotone convergence and the definition of $\theta$,
	\begin{align*}
		\theta(\mathfrak{E}^c)
			&= \lim_{n \to \infty} \int F_n\, d\theta \\
			&= \lim_{n \to \infty} \int \left[ \int F_n(g)\, \mathfrak{e}_\nu(dq) \right] \mathfrak{p}(d\nu) \\
			&= \int \left[ \int \1_{g \in \mathfrak{E}^c}\, \mathfrak{e}_\nu(dq) \right] \mathfrak{p}(d\nu) \\
			&= \int \left[ \mathfrak{e}_\nu(\mathfrak{E}^c) \right] \mathfrak{p}(d\nu) \\
			&= 0
	\end{align*}
	since $\mathfrak{e}_\nu(\mathfrak{E}^c) = 0$ for every $\nu$.
\end{proof}

\subsubsection{Part 2}
Since $\Prob(\A^\Gamma)$ is locally convex, we can assume the infimum in the definition of $\h_\Sigma^{\ann}(\mu)$ is over convex neighborhoods $\calO$. For any such $\calO$, we have
	\[ \exp \sup \{ \shent(\zeta) \st \zeta \in \bOmega(\sigma_n, \calO) \} \geq \abs*{\Omega(\sigma_n, \calO)}.\]
If $\Omega(\sigma_n, \calO)$ is empty then this is certainly true; otherwise, with $\zeta = \Unif(\Omega(\sigma_n, \calO))$ we have $\shent(\zeta) = \abs*{\Omega(\sigma_n, \calO)}$ and $\zeta \in \bOmega(\sigma_n, \calO)$ by convexity. Part 2 follows. \\

\subsubsection{Part 3}
Since $\lwlim(u,\Sigma) = \bar\mu$, for any $\calO \ni \bar\mu$ and $c>0$ for all small enough $\calU \supset \Prob^\Gamma(\calX)$ we have
	\[ \PP_{\sigma_n} \{\xi_n^\calU \not\in \bOmega(\sigma_n, \calO)\} \leq e^{-c \abs*{V_n}}  \]
for all large enough $n$.
So, using that $\xi_n^\calU$ minimizes $A$ on $\sigma_n$ among measures supported on microstates with empirical distribution in $\calU$, for any $\nu \in \Prob^\Gamma(\calX)$ and $\calO \ni \bar\mu$, for all small enough $\calU \supset \Prob^\Gamma(\calX)$ we have
\begin{align*}
	\press^{\ann}_{\Sigma}(\nu)
		&= \lim_{ \calV \downarrow \nu} \limsup_{n \to \infty} \frac{1}{\abs*{V_n}} \log \EE_{\sigma_n} \exp \sup \left\{ \widetilde{\Press}(\zeta) \st \zeta \in \bOmega(\sigma_n, \calV) \right\} \\
		&\leq \limsup_{n \to \infty} \frac{1}{\abs*{V_n}} \log \EE_{\sigma_n} \exp( \widetilde{\Press}(\xi_n^\calU) ) \\
		&= \limsup_{n \to \infty} \frac{1}{\abs*{V_n}} \log \EE_{\sigma_n} \left( \1_{\{\xi_n^\calU \in \bOmega(\sigma_n, \calO)\}} \exp( \widetilde{\Press}(\xi_n^\calU) ) + \1_{\{\xi_n^\calU \not\in \bOmega(\sigma_n, \calO)\}}\exp( \widetilde{\Press}(\xi_n^\calU) ) \right) \\
		&\leq \limsup_{n \to \infty} \frac{1}{\abs*{V_n}} \log \left( \EE_{\sigma_n} \exp \sup\{ \widetilde{\Press}(\zeta) \st \zeta \in \bOmega(\sigma_n, \calO) \} + e^{-c\abs*{V_n}} \exp( \log |\A|^{\abs*{V_n}} - \abs*{V_n}\cdot \tilde{u}^{min}  )\right) \\
		&= \max \left\{ \limsup_{n \to \infty} \frac{1}{\abs*{V_n}} \log \EE_{\sigma_n} \exp \sup\{ \widetilde{\Press}(\zeta) \st \zeta \in \bOmega(\sigma_n, \calO) \},\ \log|\A| - \tilde{u}^{min} - c \right\} . 
\end{align*}
We can then send $c$ to infinity and $\calO \downarrow \bar\mu$ to get $\press_\Sigma^{\ann}(\nu) \leq \press_\Sigma^{\ann}(\bar\mu)$.

For the second claim, note that we have just shown that for any $\nu \in \Prob^\Gamma(\calX)$
	\[ \press_\Sigma^{\ann}(\nu) \leq \lim_{\calU \downarrow \Prob^\Gamma(\calX)}\limsup_{n \to \infty} \frac{1}{\abs{V_n}} \log \EE_{\sigma_n} \exp( \widetilde{\Press}(\xi_n^\calU) ) \leq \press_\Sigma^{\ann}(\bar\mu) . \]
Taking $\nu = \bar\mu$ gives the result.

\bibliographystyle{alpha}
\bibliography{refs}

\end{document}